\documentclass[twoside,11pt]{amsart}

\usepackage{amsmath,latexsym,amssymb,mathptm, times,verbatim, enumerate,times}

\input amssym.def
\input amssym
\input xypic
\input xy

\xyoption{all}
\def\localCoh{3.2}
\def\cortolocalCohone{3.4}
\def\cortolocalCoh{3.5}
\def\SJDcortwo{5.5}

\def\notlinear{4.4.5} 
\def\unmpart{8.3} 

\def\lto{\longrightarrow}

\def\grade{\operatorname{grade}}

\def\Sym{\operatorname{Sym}}

\def\HH{\operatorname{H}}
\def\depth{\operatorname{depth}}
\def\indeg{\operatorname{indeg}}

\def\maxgendeg{b_0}

\def\L{\mathcal L}
\def\J{\mathcal J}
\def\A{\mathcal A}

\def\p{\mathfrak p}

\def\a{\mathfrak a}
\def\P{\mathfrak P}

\def\m{\mathfrak m}

\def\M{\mathfrak M}
\def\R{\mathcal R}
\def\D{\mathcal D}
\def\rank{\operatorname{rank}}

\def\htt{\operatorname{ht}}
\def\Spec{\operatorname{Spec}}
\def\ann{\operatorname{ann}}
\def\reg{\operatorname{reg}}

\def\topdeg{\operatorname{topdeg}}

\newtheorem{theorem}{Theorem}[section]
\newtheorem{lemma}[theorem]{Lemma}
\newtheorem{corollary}[theorem]{Corollary}
\newtheorem{proposition}[theorem]{Proposition}

\newtheorem{observation}[theorem]{Observation}

\newtheorem*{TheoremMainT}{Theorem \ref{MainT}}

\newtheorem{def&dis}[theorem]{Definition and Discussion}
\theoremstyle{definition}
\newtheorem{definition}[theorem]{Definition}
\newtheorem{remark}[theorem]{Remark}
\newtheorem*{Remark}{Remark}
\newtheorem*{Remarks}{Remarks}

\newtheorem{data}[theorem]{Data}

\newtheorem{chunk}[theorem]{}
\newtheorem{example}[theorem]{Example}

\numberwithin{equation}{theorem}

\begin{document}

\baselineskip=16pt

\title[The equations defining blowup algebras of 
 height three Gorenstein ideals]{ 
The equations defining blowup algebras of 
 height three Gorenstein ideals}

\author[Andrew R. Kustin, Claudia Polini, and Bernd Ulrich]
{Andrew R. Kustin, Claudia Polini, and Bernd Ulrich}

\thanks{AMS 2010 {\em Mathematics Subject Classification}.
Primary 13A30, 13H15, 14A10, 13D45; Secondary 13D02,  14E05.}

\thanks{The first author was partially supported by the Simons Foundation.
The second and third authors were partially supported by the NSF}


\address{Department of Mathematics, University of South Carolina,
Columbia, SC 29208} \email{kustin@math.sc.edu}

\address{Department of Mathematics, 
University of Notre Dame,
Notre Dame, IN 46556} \email{cpolini@nd.edu}

\address{Department of Mathematics,
Purdue University,
West Lafayette, IN 47907}\email{ulrich@math.purdue.edu}

 \begin{abstract} We find the defining equations of Rees rings of linearly presented height three Gorenstein ideals. To prove our main theorem we use local cohomology techniques to bound the maximum generator degree of the torsion submodule of  
symmetric powers in order to  conclude that the defining equations of the Rees algebra and of the special fiber ring 
generate the same ideal in the symmetric algebra. We show that 
the ideal defining the special fiber ring is the unmixed part of the ideal generated by the maximal minors of a matrix of linear forms which is annihilated by a vector of indeterminates, and otherwise has maximal possible height. An important step in the proof is the calculation of the degree of the variety parametrized by the forms generating the height three Gorenstein ideal.  \end{abstract}
 
\maketitle

\section{Introduction}

This paper deals with the algebraic study of rings that arise in the
process of blowing up a variety along a subvariety. These rings are
the Rees ring and the special fiber ring
of an ideal. 
Rees rings  
are also the bi-homogeneous coordinate rings of graphs of rational maps between projective spaces, whereas the special fiber rings are the homogeneous coordinate rings of the images of such  maps.   
More precisely,  let $R$ be a standard graded polynomial ring over a field and $I$ be an ideal of $R$ minimally generated by forms $g_1, \ldots, g_n$ of the same degree. 
These forms define a rational map $\Psi$ whose image is a variety $X$. The bi-homogeneous coordinate ring of the graph of $\Psi$ 
is the Rees algebra of the ideal $I$, which is defined as
the graded subalgebra $\mathcal R(I)=R[It]$ of the polynomial ring $R[t]$. This algebra contains, as direct summands, the ideal $I$ as well as the homogeneous 
coordinate ring of the variety $X$. 

It is a fundamental problem to find the implicit defining equations of the Rees ring and thereby of the variety $X$. 
This 
problem has been studied extensively by commutative algebraists, algebraic geometers, and, most recently, by applied mathematicians in geometric modeling. A complete solution obviously requires that one knows the structure of the ideal $I$ to begin with.
There is a large body of work dealing with perfect ideals of height two, 
which are known to be ideals of maximal minors of an almost square matrix; see for instance  \cite{HSV, M, MU, HSV08, CHW, KPU2011, Bu,  CD, KPU-BSSA, L, CD2, BM, Ma}. 
Much less is known for the `next cases', determinantal ideals of arbitrary size \cite{BCV, BCV2} and Gorenstein ideals of height three \cite{M, J}, the class of ideals that satisfies the Buchsbaum-Eisenbud structure theorem \cite{BE}. In the 
present paper we treat the case where $I$ is a height three Gorenstein ideal and the entries of a homogeneous presentation matrix of
$I$ are linear forms or, more generally, generate a complete intersection ideal; see Theorem~\ref{MainT} and Remark~\ref{general setting}. Height three Gorenstein ideals have been studied extensively, and there is a vast literature dealing with the various aspects of the subject; see for instance \cite{W, BE, EI, KU-Fam, BS, I, DV, AHH, D,  GM, MP, KM, PU, IK, CV, H, BMMNZ, KN}.

Though it may be impossible to determine the implicit equations   
of Rees rings for general classes of ideals, one can still hope
to bound the degrees of these equations, which 
is in turn 
an important step towards finding them. In the present paper we address this broader issue as well.

Traditionally  one views the  Rees algebra as a natural epimorphic image of the  symmetric algebra $\operatorname{Sym}(I)$ of $I$ and one studies the kernel of this map,
$$0\lto \mathcal A\lto \Sym(I)\lto \mathcal R(I)\lto 0.$$
The kernel  $\mathcal A$ is  the $R$-torsion submodule of $\operatorname{Sym} (I)$. The defining equations of the symmetric algebra 
can be easily read from the presentation matrix  of the ideal $I$. Hence the simplest situation is when the symmetric algebra and the Rees algebra are isomorphic, in which case the ideal $I$ is said to be of {\it linear type}. 

When  the ideal $I$ is perfect of height two, or  is Gorenstein of height three or, more generally,  is in the linkage class of a complete intersection, then $I$ is of linear type if and only if  the minimal number of generators of the ideal $I_{\p}$, $\mu(I_{\p})$, is at most  $\operatorname{dim} R_{\p}$ for every prime ideal  $\p \in V(I)$; see \cite{HSV, Hu82}. If the condition $\mu(I_{\p})\le \operatorname{dim} R_{\p}$  is only required  for every non-maximal prime ideal  $\p \in V(I)$, then   we say that $I$ satisfies $G_d$, where $d$ is the dimension of the ring $R$. The condition $G_d$ can be interpreted in terms of the height of  Fitting ideals, $i<\htt  \operatorname{Fitt}_i(I)$ whenever $i < d$. The property  $G_d$ is always satisfied if $R/I$ is $0$-dimensional or if $I$ is a generic perfect ideal of height two or a generic Gorenstein ideal of height three. For ideals in the linkage class of a complete intersection the assumption 
$G_d$ means that $\mathcal A$ is annihilated by a power of the  maximal homogeneous ideal $\mathfrak m$ of $R$, in other words $\mathcal A$ is the zeroth local cohomology $\HH^0_{\mathfrak m}(\operatorname{Sym}(I))$. One often requires the condition $G_d$ when studying Rees rings of ideals that are not necessarily of linear type. 

To identify further implicit equations of the Rees ring, in addition to the equations defining the symmetric algebra, one uses the technique of   {\it Jacobian dual}; see for instance \cite{V}. Let  $x_1, \ldots, x_d$ be the variables of the polynomial ring $R$ and $T_1, \ldots, T_n$ be  new variables. Let  
$\varphi$ be  a minimal homogeneous presentation matrix of the forms  $g_1, \ldots, g_n$ generating $I$. 
There exits a matrix $B$ with $d$ rows and entries in the polynomial ring $S=R[T_1, \ldots, T_n]$ such that the equality of row vectors
\[[T_1, \ldots, T_n] \cdot \varphi= [x_1, \ldots, x_d]\cdot B\]
holds. Choose $B$ so that its entries are linear in $T_1, \ldots, T_n$. The matrix $B$ is called a {\it Jacobian dual} of $\varphi$. If $\varphi$ is a matrix of linear forms, then the entries of $B$ can be taken from the ring $T=k[T_1, \ldots, T_n]$; this $B$ is uniquely determined by $\varphi$ and is called
{\it the} Jacobian dual of $\varphi$. The choice  of generators of $I$ gives rise to homogeneous epimorphisms of graded $R$-algebras
\[\xymatrix{
& S\ar@{->>} [r]& \operatorname{Sym} (I) \ar@{->>}[r]& \mathcal R(I)\, .
}\]
We denote by $\mathcal L$ and $\mathcal J$ the ideals of $S$ defining $ \operatorname{Sym} (I)$ and $\mathcal R(I)$, respectively. 
The ideal $\mathcal L$ is generated by the entries of the row vector  $[T_1, \ldots, T_n] \cdot \varphi=[x_1, \ldots, x_d]\cdot B$. Cramer's rule shows that 
\[\mathcal L +I_d(B) \subset \mathcal J\, ,\]
and if equality holds, then $\mathcal J$ or $\R(I)$ is said to have the {\it expected form}. The expected form is 
the next best possibility 
for the defining ideal of the Rees ring  if $I$ is not of linear type. Linearly presented ideals whose Rees ring has the expected form are 
special cases of ideals 
of fiber type. An ideal $I$ is of {\it fiber type} if the defining ideal $\J$ of $\R(I)$ can be reconstructed from the defining ideal $I(X)$
of the image variety $X$ or, more precisely,
$$ \J=\L+I(X)S\, .$$ 

It has been shown in \cite{MU} that the Rees ring has the expected form if $I$ is a linearly presented height two perfect ideal satisfying $G_d$. In this situation the Rees algebra is Cohen-Macaulay. There has been a great deal of work investigating the defining ideal of Rees rings when $I$ is a height two perfect ideal that either fails to satisfy $G_d$ or  is not linearly presented. In this case $\mathcal R(I)$ usually does not have the expected form and  is  not Cohen-Macaulay; see for instance  \cite{HSV08, CHW, KPU2011, Bu,  CD, KPU-BSSA, L, CD2, BM, Ma}.

Much less is known for height three Gorenstein ideals. In this case $n$ is odd and the presentation matrix $\varphi$ can be chosen to be
alternating, according to the Buchsbaum-Eisenbud structure theorem \cite{BE}.  The defining ideal $\mathcal J$ of the Rees ring has been determined provided that  $I$  satisfies $G_d$,  $n=d+1$, and moreover $\varphi$ has linear entries \cite{M} or, more generally, $I_1(\varphi)$ is generated by  the entries of a single row of $\varphi$ \cite{J}. It turns out that $\mathcal R(I)$ does not have the expected form, but is Cohen-Macaulay under these hypotheses. On the other hand, 
this is the only case when the Rees ring is Cohen-Macaulay. In fact, the Rees ring of a height three Gorenstein ideal satisfying $G_d$  is Cohen-Macaulay if and only if  either $n \le d$ or else $n=d+1$  and $I_1(\varphi)$ is generated by  the entries of a single generalized row of $\varphi$ \cite{PU}. Observe that if $n=d+1$ then $d$ is necessarily even. The `approximate resolutions' of the symmetric powers 
of $I$ worked out in \cite{KU-Fam} show that if $n \ge d+1$  and $d$ is even,  then $\Sym (I)$ has $R$-torsion in too low a degree for $\mathcal J$ to have the expected form; see also \cite{M}. As it turns out, the alternating structure of the presentation matrix $\varphi$  is responsible for 
`unexpected' elements in $\J$. Let $\varphi$ an alternating matrix of linear forms presenting a height three Gorenstein ideal, and let $B$ be its
Jacobian dual.  As in \cite{KU-Fam} one builds an  alternating matrix from $\varphi$ and $B$,
$$\goth B=\bmatrix \varphi&-B^{\rm t}\\B&0\endbmatrix\,.$$
The submaximal Pfaffians of this matrix are easily seen to belong to $\J$. 
We regard the last of these Pfaffians 
as
a polynomial in $T[x_1, \ldots, x_d]$ and consider its content ideal $C(\varphi)$ in $T$. The main theorem of the present paper says that $C(\varphi)$, together with the expected equations, generates the ideal $\J$:

\begin{TheoremMainT} Let $I$ be a linearly presented height three Gorenstein ideal that  satisfies $G_d$. Then the defining ideal of the Rees ring of $\,I$ is 
$$\J=\L+I_d(B)S+C(\varphi)S$$
and the defining ideal of  the variety $X$  is 
$$I(X)=I_d(B)+C(\varphi),$$ 
where $\varphi$ is a minimal homogeneous alternating presentation matrix for $I$.
In particular,  $I$ is of fiber type and,  if $d$ is odd, then $C(\varphi)=0$ and $\mathcal R(I)$ has the expected form.
\end{TheoremMainT}

We now outline the method of proof of our main theorem. The proof has essentially four ingredients, each of which is of independent interest and  applies in much greater generality than needed here. Some of these tools are developed in other articles \cite{KU-Fam, KPU-HS, KPU-DBLC, KPU-ann}, and we
are going to describe the content of these articles to the extent necessary for the present paper. 

\smallskip

The first step is to prove that any ideal $I$ satisfying the assumptions of Theorem \ref{MainT} is of fiber type (Corollary~\ref{fiber type}). Once this
is done, it suffices to determine $I(X)$. A crucial tool in this context is the standard bi-grading of the symmetric algebra. If $\delta$
is the degree of the generators of $I$, then the $R$-module $I(\delta)$ is generated in degree zero and its symmetric algebra $\Sym(I(\delta))$ is naturally standard bi-graded. As it turns out, the ideal $I$ is of fiber type if and only if the bi-homogeneous ideal $\A=\HH^0_{\mathfrak m}(\operatorname{Sym}(I(\delta)))$ of the symmetric algebra  is generated in degrees $(0, \star)$ or, equivalently, the $R$-modules $\HH^0_{\mathfrak m}(\operatorname{Sym}_q(I(\delta)))$ are generated in degree zero for all non-negative integers $q$.  This suggests the following general question: 
If $M$ is  a finitely generated graded $R$-module, what are bounds for the generator degrees of the local cohomology modules $\HH^i_{\m}(M)$? This question is addressed in \cite{KPU-DBLC}. There we consider {\it approximate resolutions}, which are homogeneous complexes of finitely
generated graded $R$-modules $C_j$ of  sufficiently high  depth,
$$\, C_{\bullet} :\qquad  \ldots  \ \lto C_1 \lto C_0 \lto 0\, ,$$
such that  $\HH_0(C_{\bullet}) \cong M$ and $\HH_j(C_{\bullet})$ have  sufficiently small dimension for all positive integers $j$. 
We prove, for instance, that  $\HH^0_{\m}(M)$ is concentrated in degrees $\le b_0(C_d)-d$ and is generated in degrees $\le b_0(C_{d-1}) -d+1$, where $b_0(C_j)$ denotes the largest generator degree of $C_j$.  These theorems from  \cite{KPU-DBLC}  are reproduced in the present paper as Corollary~\ref{cor-to-localCoh-1} and Theorem~\ref{SJD-cor2}.
 Whereas similar results for the concentration degree of local cohomology modules have been established before to study Castelnuovo-Mumford regularity (see, e.g., \cite{GLP}), it appears that the sharper bound on the generation degree is new and more relevant for our purpose. Under the hypotheses of 
Theorem~\ref{MainT} the complexes $\mathcal D^q_{\bullet}$ from \cite{KU-Fam} are approximate resolutions of the symmetric powers $\Sym_q(I(\delta))$, and we deduce that indeed the $R$-modules $\HH^0_{\m}({\rm Sym}_q(I(\delta)))$ are generated in degree zero. This is done in Theorem~\ref{App-to-BA-1} and Corollary~\ref{fiber type}. 

In fact, Theorem~\ref{App-to-BA-1} is  considerably more general; it deals with height two perfect ideals and height three Gorenstein ideals 
that are not necessarily linearly presented. 
We record an explicit exponent $N$ so that $\mathfrak m^N\cdot \mathcal A=0$ and we provide an explicit bound  for the largest $x$-degree $p$ so that $\mathcal A_{(p,*)}$ contains a minimal bi-homogeneous generator of $\mathcal A$.    
Theorem~\ref{App-to-BA-1}  is a  
significant generalization 
of part of  \cite{KPU-BSSA},
where the same results are established under fairly restrictive hypotheses. 

\smallskip

With the hypotheses of Theorem \ref{MainT} we know that $I_d(B)+C(\varphi)  \subset I(X)$; see \cite{KPU-ann} or
Corollary~\ref{content} in the present paper. To show that this inclusion is an equality, we first prove that the two ideals have the same height, which is $n-d$ if $d \le n$. In fact, we show much more under weaker hypotheses: Assume $I$ is any ideal of $R$ generated by forms of  degree $\delta$, $I$ is of linear type on
the punctured spectrum, and a sufficiently high symmetric power $\Sym_q(I(\delta))$ has an approximate free resolution that is linear for the first $d$ steps, then up to radical, $\J$ has the expected form, in symbols, $\J=\sqrt{(\L, I_d(B))}$ and, as a consequence,  $I(X)=\sqrt{I_d(B')}$, where $B'$ is the Jacobian dual of the largest submatrix of $\varphi$ that has linear entries. This is Theorem~\ref{correct grade} in the present paper. We point out that even though our primary interest was the defining ideal $I(X)$ of the variety $X$, we had to consider the entire bi-graded symmetric algebra of $I$ since our main assumption uses
the first component of the bi-grading, which is invisible to the homogeneous coordinate ring of $X$.

\smallskip

The third step in the proof of Theorem \ref{MainT} is to show that  the ideal $I_d(B)+C(\varphi)$ is unmixed, in the relevant case when $d<n$. In \cite{KPU-ann} we construct, more generally, complexes associated to any $d$ by $n$ matrix
$B$ with linear entries in $T$ that is annihilated by a vector of indeterminates, $B\cdot {\underline T}^t=0$. The ideal of $d$ by $d$ minors of such a matrix cannot
have generic height $n-d+1$. However,
if the height is $n-d$, our complexes give resolutions of $I_d(B)$ and of $I_d(B)+C(\varphi)$, although these ideals fail to be perfect in general. 
Thus, we obtain resolutions in the context of Theorem \ref{MainT} because we have seen in the previous step that $I_d(B)$ has height $n-d$.
Using these resolutions we 
prove that the ideal $I_d(B)+C(\varphi)$ is unmixed and hence is the unmixed part of $I_d(B)$. 
In \cite{KPU-ann} we also compute 
the multiplicity of the rings defined by these ideals. 

\smallskip

We have seen that the two ideals $I_d(B)+C(\varphi)$ and $I(X)$ are unmixed and have the same height.
Thus, to conclude that the inclusion $I_d(B)+C(\varphi)\subset I(X)$ is an equality it suffices to prove that the rings defined by these ideals have the same multiplicity. Since the multiplicity of the first ring has been computed in \cite{KPU-ann}, it remains to determine  the multiplicity of the
coordinate ring $A=T/I(X)$ or, equivalently,  the degree of the variety $X$. To do so we observe that the rational map $\Psi$ is birational onto its image because the presentation matrix $\varphi$ is linear (Proposition~\ref{regular}). Therefore the degree of $X$ can be expressed in terms of the multiplicity of a ring defined by a certain residual intersection of $I$ (Proposition~\ref{xie}). Finally, the multiplicity of such residual intersections can be obtained from the resolutions in \cite{KU-Fam}; see \cite{KPU-HS} and Theorem~\ref{MULT} in the present paper. Once  we  prove in Theorem~\ref{MainT} that  $I(X)=I_d(B)+C(\varphi)$, then in Corollary~\ref{Cor-to-main} 
we are able to harvest much information from the results of \cite{KPU-ann}. 
In particular, we learn the depth of the homogeneous coordinate ring $A$ of $X$, the entire Hilbert series of $A$, and the fact that $I(X)$ has a linear resolution when $d$ is odd.


\section{Conventions and notation}\label{prelim}

\begin{chunk} If $\psi$ is a matrix, then $\psi^{\rm t}$ is the transpose of $\psi$. If $\psi$ is a matrix (or a homomorphism of finitely generated free $R$-modules), then $I_r(\psi)$ is the ideal generated by the
$r\times r$ minors of $\psi$ (or any matrix representation of $\psi$).\end{chunk}

\begin{chunk}\label{numerical-functions}We collect  names  for some of the invariants associated to a graded module. Let $R$ be a graded Noetherian ring and $M$ be a graded $R$-module. Define
\begin{align}
\topdeg M&=\sup\{j \mid M_j\not=0\},
\notag\\
\indeg M&=\inf\{j \mid M_j\not=0\}, \text{ and}
\notag\\
b_0(M)&\textstyle=\inf \left\{p \, | \,  R\left(\bigoplus_{j\le p}M_j\right)=M \right\}.\notag\\
\intertext{
If $R$ is non-negatively graded, $R_0$ is local,  $\mathfrak m$ is the maximal homogeneous ideal of $R$, $M$ is finitely generated, and $i$ is a non-negative integer, then also define}
a_i(M)&=\topdeg \HH_{\mathfrak m}^i(M) \text{ \ and}
\notag\\
b_i(M)&=\topdeg \operatorname{Tor}^R_i(M,R/\mathfrak m).
\notag\end{align}
Observe that both definitions of the maximal generator degree $b_0(M)$ give the same value.  The expressions ``$\topdeg$'', ``$\indeg$'', and ``$b_i$'' are read ``top degree'', ``initial degree'', and 
``maximal $i$-th shift in a minimal homogeneous resolution'',
respectively.
If $M$ is the zero module, then $$\topdeg(M)=b_0(M)=-\infty\quad\text{and}\quad \indeg M=\infty.$$ In general one has 
$$a_i(M)< \infty \quad\text{and}\quad b_i(M)< \infty\, .$$
\end{chunk}

\begin{chunk}Recall the numerical functions of \ref{numerical-functions}.  Let $R$ be a non-negatively graded Noetherian ring with $R_0$ local and  $\mathfrak m$ be the maximal homogeneous ideal of $R$. 
The $a$-invariant of $R$ is defined to be
%
%
$$a(R)=a_{\dim R}(R).$$
Furthermore, if $\omega_R$ is the graded canonical module of $R$
, then 
$$a(R)=-\indeg \omega_R
.$$
\end{chunk}

\begin{chunk}
\label{aiM} Let  $R$ be a standard graded polynomial ring over a field $k$, $\mathfrak m$ be the maximal homogeneous ideal of $R$, and   
$M$ be a finitely generated graded $R$-module.  Recall the numerical functions of \ref{numerical-functions}. 
The Castelnuovo-Mumford {\it regularity} of  $M$ is
$$\reg M=\max\{a_i(M)+i\}=\max\{b_i(M)-i\}.$$
\end{chunk}

\begin{chunk} Let $I$ be an ideal in a Noetherian ring. The {\it unmixed part} $I^{\text{unm}}$  of $I$ is the intersection of the primary components 
 of $I$ that have height equal to the height of $I$. \end{chunk}

\begin{chunk} Let $R$ be a Noetherian ring,  $I$ be a proper ideal of $R$, and $M$ be a finitely generated non-zero $R$-module. 
We denote the projective dimension of the $R$-module $M$ by  $\operatorname{pd}_R(M)$.  The {\it grade} of $I$ is the length of a maximal regular sequence on $R$ which is 
contained in $I$. (If $R$ is Cohen-Macaulay, then the grade
of $I$ is equal to the height of $I$.) The $R$-module $M$ is called {\it perfect} if the grade of the annihilator of $M$ (denoted $\operatorname{ann} M$) is equal to the projective dimension of $M$. (The inequality $\grade\operatorname{ann} M\le \operatorname{pd}_R M$ holds automatically.) The ideal $I$  in  $R$ is called a {\it perfect ideal} if $R/I$ is a perfect $R$-module.
 A perfect ideal $I$ of grade $g$ is a {\it Gorenstein ideal}
 if $\operatorname{Ext}^g_R(R/I,R)$ 
is a cyclic $R$-module.
\end{chunk}

 \begin{chunk}\label{Gd-def}
Denote by $\mu(M)$ the minimal number of generators of a finitely generated module $M$ over a local ring $R$. Recall from Artin and Nagata \cite{AN} that an ideal $I$ in a Noetherian ring $R$ satisfies the condition  $G_s$ if $\mu(I_{\mathfrak p}) \le \operatorname{dim}  R_{\mathfrak p}$ for each prime ideal $\mathfrak p \in V(I)$  with $\operatorname{dim} R_{\mathfrak p} \le s-1$.
The condition $G_s$ can be interpreted in terms of the height of  Fitting ideals. 
An ideal $I$ of positive height satisfies $G_s$ if and only if  $i<\htt \, \operatorname{Fitt}_i(I)$ for $0 < i < s$. 
 
 The property  $G_d$, for $d={\rm dim}\, R$, is always satisfied if $R/I$ is $0$-dimensional or if $I$ is a generic perfect ideal of grade two or a generic Gorenstein ideal of grade three.

\end{chunk}

\begin{chunk} Let $R$ be a Noetherian ring,  $I$ be an ideal of height $g$,
$K$ be a proper ideal, and  $s$   be an integer with $g\le s$.
\begin{enumerate}[\rm (1)]
\item[(a)]
The ideal $K$ is called an {\it $s$-residual intersection} of $I$ if there exists an $s$-generated ideal $\a \subset I$
such that $K=\a : I$ and $s\le \operatorname{ht} K$.
\item[(b)]
The ideal $K$ is called a {\it geometric $s$-residual intersection} of $I$ if $K$ is an $s$-residual intersection
of $I$ and if in addition $s+1\le \htt (I+K)$.
\item[(c)]
The ideal $I$ is said to be {\it weakly $s$-residually $S_2$} if for every $i$ with $g\le i \le s$ and every geometric $i$-residual intersection $K$ of $I$, $R/K$ is $S_2$.
\end{enumerate}
\end{chunk}

\begin{chunk}\label{REES} If $I$ is an ideal of a ring $R$, then the {\it Rees ring} of $I$, denoted $\mathcal R(I)$, is the subring $R[It]$ of the polynomial ring $R[t]$. There is a natural epimorphism of $R$-algebras from the symmetric algebra of $I$, denoted $\Sym(I)$, onto $\R(I)$. If this map is an isomorphism, we say that the ideal $I$ is of {\it linear type}.
\end{chunk}
\medskip

\section{Defining equations of graphs and images of rational maps}\label{Intro-B-A}

\begin{data}\label{data1} Let $k$ be a field and let $R=k[x_1,\dots,x_d]$, $T=k[T_1,\dots,T_n]$, and $S=R[T_1,\dots,T_n]$ be polynomial rings. Let $\mathfrak m=(x_1, \ldots, x_d)$ be the  maximal homogeneous ideal of $R$, and $I$ a homogeneous ideal of $R$ minimally generated by forms $g_1, \ldots, g_n$ of the same positive degree $\delta$. 

Such forms define a rational map \[\Psi=[g_1: \ldots: g_n]:\xymatrix{\mathbb P_k^{d-1}\ar@{-->}[r]&\mathbb P_k^{n-1}}\, \]
with base locus $V(I)$. We write $X$ for the closed image of $\Psi$, the variety parametrized by $\Psi$, and $A=T/I(X)$ for the homogeneous coordinate ring of  this variety. 
\end{data} 

\smallskip

In this section we collect some general facts about rational maps and Rees rings.
Data \ref{data1} is in effect throughout.

\smallskip

\begin{chunk}\label{grading} Recall the definition of the Rees ring $\R(I)=R[It]\subset R[t]$ from \ref{REES}. We consider the epimorphism of $R$-algebras
$$\pi: \xymatrix{S\ar@{->>}[r]&\mathcal R(I)}$$
that sends $T_i$ to $g_it$. Let $\J$ be the kernel of this map, which is the ideal defining the Rees ring. 

We set ${\deg x_i=(1,0)}$, $\deg T_i=(0,1)$, and $\deg t=(-\delta, 1)$. Thus $S$ and the subring $\R(I)$ of $R[t]$ become standard bi-graded $k$-algebras, $\pi$ is a bi-homogeneous epimorphism, and $\J$ is a bi-homogeneous ideal of $S$. If $M$ is a bi-graded $S$-module and $p$ a fixed integer, we write 
$$ M_{(p, \star)}=\bigoplus_j \, M_{(p,j)} \ \ \  {\rm and} \ \  \  M_{(> p, \star)}=\bigoplus_{ i >p} \, M_{(i,\star)} \, .$$
\end{chunk}
 
\vspace{.0cm}

\begin{chunk}\label{graph}  The Rees algebra $\R(I)$ is the bi-homogeneous coordinate ring of the graph of $\Psi$. In fact, the  natural morphisms of projective varieties
$$\xymatrix{ \mathbb P_k^{d-1} \times \mathbb P_k^{n-1}\ar@{->>}[d] &\ {\rm graph} \, \Psi \ar@{->>}[d]  \ \ar@{_{(}->}[l] \\
\  \ \mathbb P_k^{n-1} &\ X={\rm im} \, \Psi  \ \ar@{_{(}->}[l] }$$
correspond to bi-homogeneous homomorphisms of $k$-algebras
$$\xymatrix{
S\ar@{->>}[r]^{\pi} & \mathcal R(I)\\
T\ar@{->>}_{\pi_{|T}}[r]   \ar@{^{(}->}[u] &\ A   \ar@{^{(}->}[u]\, .
}$$
Since $T=S_{(0, \star)}$, we see that 
$$A=\R(I)_{(0, \star)}=k[g_1t, \ldots, g_nt]$$
and 
\begin{equation}\label{I of X}I(X)=\J_{(0,\star)}=\J \cap T\, .
\end{equation}

We also observe that there is a not necessarily homogeneous isomorphism of $k$-algebras 
$$k[g_1t, \ldots, g_nt] \cong k[g_1, \ldots, g_n]\subset R\, .$$
Moreover, the identification $\R(I)_{(0,\star)}\cong \R(I)/\R(I)_{(>0,\star)}$ gives a homogeneous isomorphism
$$A\cong\R(I)/\m\R(I)\,.$$
The latter ring is called the {\it special fiber ring} of $I$ and its dimension is the {\it analytic spread} of $I$, written $\ell(I)$. The analytic spread plays an important role in the study of reductions of ideals and satisfies the inequality 
\begin{equation}\label{ell}\ell(I)\le \min\{d, n\}\, .\end{equation}
Notice that $\ell(I)=\dim A=\dim X+1$. 
\end{chunk}

\smallskip

\begin{chunk}\label{LT-FT} The graded $R$-module $I(\delta)$ is generated in degree zero and hence its symmetric algebra $\Sym(I(\delta))$ is a standard bi-graded $k$-algebra. There is a natural bi-homogeneous epimorphism of $k$-algebras 
$$ \xymatrix{\Sym(I(\delta))\ar@{->>}[r]&\mathcal R(I)\, ,} $$
whose kernel  we denote by $\mathcal A$. The epimorphism $\pi$ from \ref{grading} factors through this map and gives a bi-homogeneous epimorphism of $k$-algebras
\[\xymatrix{
& S\ar@{->>} [r]& \operatorname{Sym} (I(\delta)) \,.
}\]
We write $\L$ for the kernel of this epimorphism. The advantage of $\L$ over $\J$ is that $\L$ can be easily described in terms of a presentation matrix of $I$; see \ref{def1}. 
Notice that 
$$\A=\J/\L\, .$$

\noindent
The ideal $I$ is  of  linear type if any of the following equivalent conditions holds:
\begin{itemize}
\item $\A=0$
\item $\J=\L$
\item $\J$ is generated in bi-degrees $(\star, 1)$. \end{itemize}

\noindent
Notice that if $I$ is of linear type then necessarily $I(X)=0$. 

\smallskip
\noindent
The ideal $I$ is said to be of {\it fiber type} if any of the following equivalent conditions holds:

\begin{itemize}
\item $\A=I(X)\cdot \Sym(I(\delta))$
\item $\A$ is generated in bi-degrees $(0,\star)$
\item $\J=\L+I(X)\cdot S$
\item $\J$ is generated in bi-degrees $(\star, 1)$ and $(0,\star)$;
\end{itemize}
see (\ref{I of X}) for the equivalence of these conditions.
\end{chunk}
 
 \vspace{0.0cm}
 
 \begin{chunk}

The ideal $\mathcal A$ in the exact sequence
$$0\lto \mathcal A\lto \Sym(I(\delta))\lto \mathcal R(I)\lto 0$$
is  the $R$-torsion  of $\Sym (I(\delta))$. Hence $\A$ contains the zeroth local cohomology of $\Sym(I(\delta))$ as an $R$-module, which in turn contains the socle of $\Sym(I(\delta))$ as an $R$-module,
\begin{equation}\label{sym}\A \supset \HH^0_{\m}({\Sym}(I(\delta)))=0:_{\Sym(I(\delta))} \m^{\infty}\supset 0:_{\Sym(I(\delta))} \m\, .\end{equation}
\end{chunk}
 
 \vspace{0.0cm}
 
\begin{chunk}\label{def1}Consider the row vectors  $\underline{x}=[x_1, \ldots, x_d]$, $\underline{T}=[T_1, \ldots, T_n]$, $\underline{g}=[g_1, \ldots, g_n]$, and let $\varphi$ be  a minimal homogeneous presentation matrix of $\underline{g}$. 
There exits a matrix $B$ with $d$ rows and entries in the polynomial ring $S$ such that the equality of row vectors
\begin{equation}\label{3.2.1}\underline{T} \cdot \varphi= \underline{x}\cdot B\end{equation}
holds. If $B$ is chosen so that its entries are linear in $T_1, \ldots, T_n$, then $B$ is called a {\it Jacobian dual} of $\varphi$. 
We say that a Jacobian dual is  {\it homogeneous} if the entries of any fixed column of $B$ are homogeneous of the same bi-degree.   Whenever $\varphi$ is a matrix of linear forms, then the entries of $B$ can be taken from the ring $k[T_1, \ldots, T_n]$; this $B$ is uniquely determined by $\varphi$ and is called
{\it the} Jacobian dual of $\varphi$; notice that the entries of $B$ are linear forms and that $B$ is indeed a Jacobian matrix of the $T$-algebra $\Sym(I(\delta))$.

The ideal $\mathcal L$ defining the symmetric algebra is equal to $I_1(\underline{T} \cdot \varphi)$, which coincides with $I_1(\underline{x}\cdot B)$. 
For the Rees algebra, the inclusions 
\begin{equation}\label{6.2.1}\L +I_d(B) \subset \L:_S \m\subset  \mathcal J\end{equation}
obtain; see also \cite{V}. If the equality $\L+I_d(B)=\J$ holds, then $\mathcal J$ or $\R(I)$  is said to have the {\it expected form}. Notice that in this case the inclusions of (\ref{sym}) are equalities, in other words, $\A$ is the socle of $\Sym(I(\delta))$ as an $R$-module. 
\end{chunk}
\begin{chunk}\label{Intro-to-BA}  It is much easier to determine when the first inclusion of (\ref{sym}) is an equality. Namely, 
\begin{equation}\label{A=H}\A=\HH^0_{\m}({\Sym}(I(\delta))) \ \Longleftrightarrow \  I  \mbox{ is of linear type 
on the punctured spectrum of } R\, . \end{equation}
The second condition means that $\mathcal A_\mathfrak p=0$ for all $\mathfrak p$ in $\operatorname{Spec}(R)\setminus \{\m\}.$

This leads to the question of when an ideal is of linear type. Thus let $J$ be an ideal of a local Cohen-Macaulay ring of dimension $d$. If $J$ is perfect of height two or Gorenstein of height three, then $J$ is in the linkage class of a complete intersection \cite{G, W}. If $J$ is in the linkage class of a complete intersection, then $J$ is strongly Cohen-Macaulay \cite[1.11]{Hu82} and hence  satisfies the sliding depth condition of \cite{HVV}.
If $J$ satisfies the sliding depth condition, then $J$ is weakly $(d-1)$-residually $S_2$ in the sense of 2.8.c \cite[3.3]{HVV}.
In the presence  of the weak $(d-2)$-residually $S_2$ condition it follows from \cite[3.6(b)]{CEU} that
$$J \mbox{ is of linear type  } \ \Longleftrightarrow \   J \mbox{ satisfies  } G_{d+1}\, ;$$
in particular, 
$$J \mbox{ is of linear type on the punctured spectrum } \ \Longleftrightarrow \   J \mbox{ satisfies  } G_{d}\, .$$
(Alternatively, one uses \cite{HSV}.)

Combining these facts we see that if the ideal $I$ of Data \ref{data1} is perfect  of height two or Gorenstein of height three then
\begin{equation}\label{Aug232016} I \mbox{  is of linear type }  \ \Longleftrightarrow \  I  \mbox{ satisfies } G_{d} \mbox{ and } n\le d \, 
\end{equation}
and 
\begin{equation}\label{apr30'} \A=\HH^0_{\m}({\Sym}(I(\delta))) \ \Longleftrightarrow \  I  \mbox{ satisfies } G_{d}\, .
\end{equation}
 \end{chunk}
\bigskip

\section{The degree of the image of a rational map}
\medskip

This section deals with the dimension and the degree of the image of rational maps as in Data {\rm\ref{data1}}. 
Applied to the special case of linearly presented Gorenstein ideals of height three, this information will be an important
ingredient in the proof of Theorem~\ref{MainT}.

We first treat the question of when the rational map $\Psi$ is birational onto its image $X$ (Proposition~\ref{regular}). We then express the degree of $X$ in terms of the multiplicity of a ring defined by a residual intersection  of the ideal $I$ defining the base locus of $\Psi$ (Proposition~\ref{xie}). The multiplicity of rings defined by residual intersections of linearly presented height three Gorenstein ideals was computed in \cite{KPU-HS}, using the resolutions worked out in \cite{KU-Fam}. For completeness we state this result in Theorem~\ref{Main-C}. Finally, we combine Proposition~\ref{regular}, Proposition~\ref{xie}, and Theorem~\ref{Main-C} in Corollary~\ref{MULT} to compute the degree of $X$ under the hypotheses of Theorem~\ref{MainT}. 

\begin{proposition}\label{regular} Adopt Data {\rm\ref{data1}}.  Assume further that the ideal $I$ is linearly presented.
\begin{enumerate}[\rm(a)]
\item\label{3.5.a} If 
the ring $R/I$  has dimension zero, then the map $\Psi$ is biregular onto its image. In particular, 
$A$ has multiplicity $\delta^{d-1}$. 
\item\label{3.5.b} If
$A$ has dimension $d$, then the rational map $\Psi$ is birational onto its image.
 \end{enumerate}\end{proposition}

\begin{proof} We may assume that the field $k$ is algebraically closed. Let $\varphi$ be a minimal presentation matrix of $\underline{g}=g_1, \ldots, g_n$ with homogeneous linear entries  in $R$ and let 
$P\in \mathbb P_k^{n-1}$ be a point in $X$. According to \cite{EU} the ideal $I_1(P \cdot \varphi) : I^{\infty}$ defines the fiber $\Psi^{-1}(P)$ scheme theoretically. This ideal cannot be the unit ideal because the fiber is not empty. On the other hand, the ideal $I_1(P \cdot \varphi)$  is generated by linear forms, hence is a prime ideal of $R$. Thus  $I_1(P \cdot \varphi) : I^{\infty}=I_1(P \cdot \varphi)$ is generated by linear forms. 

(\ref{3.5.a}) We consider the embedding $\, A \cong k[\underline g] \subset R \,$ 
which corresponds to the rational map $\Psi$ onto its image. If the ring $R/I=R/(\underline g)$ has dimension zero, then $R$ is finitely generated as an $A$-module. It follows that the fiber $\Psi^{-1}(P)$ 
consists of finitely many points. Since this fiber is defined by a linear ideal, it consists of a single reduced point. This shows that $\Psi$ is biregular onto its image.

The claim about the multiplicity follows because, in particular, $\Psi$ is birational onto its image; see, for instance, \cite[5.8]{KPUB}.

(\ref{3.5.b}) Let $Q$ be a general point in $\mathbb P_k^{d-1}$ and $P$ be its image $\Psi(Q)$ in $X$. Since $A \cong k[\underline g]$ and $R$ have the same dimension, the fiber $\Psi^{-1}(P)$ consists of finitely many points. Therefore again,  $\Psi^{-1}(P)$ consists of a single reduced point, which means that $\Psi$ is birational onto its image. 
\end{proof}

\smallskip

\begin{proposition}\label{xie} Adopt Data {\rm\ref{data1}}. Further assume that the field $k$ is infinite, that $d \le n$, and that the ideal 
$I$ is weakly $(d-2)$-residually $S_2$ and satisfies $G_d$. 
Let $f_1, \ldots, f_{d-1}$ be general $k$-linear combinations of the generators $g_1, \ldots, g_n$ of $\, I$. Then the following statements hold.
\begin{enumerate}[\rm(a)]
\item\label{3.6.a} The ring $\frac{R}{(f_1, \ldots, f_{d-1})\, :\, I}$ is Cohen-Macaulay of dimension one.
\item\label{3.6.b} The ring $A$ has dimension $d$.
\item\label{3.6.c} The ring $A$ has multiplicity
\[e(A)=\frac{1}{r} \, \cdot \, e\left(\frac{R}{(f_1, \ldots, f_{d-1}):I}\right)\, , \]where  $r$ is the degree of the rational map $\Psi$. 
\end{enumerate}
\end{proposition}
\begin{proof} Let  $f_1, \ldots, f_{d}$ be general $k$-linear combinations of the generators $g_1, \ldots, g_n$ of $I$. The $G_d$ property of the ideal $I$ implies that 
\begin{equation}\label{RI}
d-1\le \htt ( (f_1, \ldots, f_{d-1}) :I) \qquad \mbox{ and } \qquad d\le \htt ((f_1, \ldots, f_{d-1}) :I, f_{d})
\end{equation}
 according to \cite[2.3]{AN}  or \cite[1.6(a)]{UAN}. Furthermore the ideal  $(f_1, \ldots, f_{d-1}):I$ is proper because $I$ requires at least $d$ generators. Now \cite[3.1 and 3.3(a)]{CEU} implies that $(f_1, \ldots, f_{d-1}):I$ is unmixed of height $d-1$, which proves (\ref{3.6.a}). 

To show (\ref{3.6.b}) we suppose that $\operatorname{dim} A <d$. In this case the ring $\, k[g_1, \ldots, g_n] \cong A \, $ is integral over the subring  $k[f_1, \ldots, f_{d-1}]$, and therefore the ideal $I$ is integral over the ideal $(f_1, \ldots, f_{d-1})$.  In \ref{Intro-to-BA} we have seen that,
locally on the punctured spectrum, $I$ is of linear type  and hence  cannot be integral  over a proper subideal. Thus $d\le \htt ( (f_1, \ldots, f_{d-1}) :I)$, which contradicts (\ref{3.6.a}). 

We now prove (\ref{3.6.c}).  Item (\ref{3.6.b}) allows us to apply  \cite[3.7]{KPUB},  which gives

\[e(A)=\frac{1}{r}\, \cdot \,  e\left(\frac{R}{(f_1, \ldots, f_{d-1}):I^{\infty}}\right)\, . \]

\noindent
On the other hand, item (a) and  the inequality  $d\le \htt ((f_1, \ldots, f_{d-1}) :I, f_{d})$  imply that the element $f_d$  of $I$ is regular on  $ \frac{R}{(f_1, \ldots, f_{d-1})\, :\, I}$. It follows that $(f_1, \ldots, f_{d-1}):I^{\infty}=(f_1, \ldots, f_{d-1}):I$. 
\end{proof}

\bigskip
The following statement, which is used in the proof of Theorem~\ref{MULT}, is part of \cite[1.2]{KPU-HS}. We do not need the entire  $h$-vector of $\overline R$ in the present paper; but the entire $h$-vector  is  recorded in \cite[1.2]{KPU-HS}.
\begin{theorem}\label{Main-C} Adopt Data {\rm\ref{data1}}. Further assume that the ideal 
$I$ is a linearly presented Gorenstein ideal of height three.
 Let $\goth a$ be a subideal of $I$ minimally  generated by $s$ homogeneous elements for some $s$ with $3\le s$, $\alpha$ be the minimal number of generators of $I/\goth a$, $J$ be the ideal $\goth a:I$, and $\overline R=R/J$.
Assume that  \begin{enumerate}[\rm(1)]
\item the homogeneous minimal generators of $\goth a$ live in two degrees$:$ $\indeg(I)$ and $\indeg(I)+1$,
\item the ideal $J$ is an $s$-residual intersection of $I$ {\rm(}that is, $J\not=R$ and $s\le \htt J${\rm)}.
\end{enumerate}
Then the ring $\overline R$ has  multiplicity
$$e({\overline R})=\sum_{i=0}^{\lfloor\frac {\alpha-1}2\rfloor}\binom{s+\alpha-2-2i}{s-1}.$$
The latter  is also equal to the number of monomials $\bf m$ of degree at most $\alpha-1$ in $s$ variables with $(\deg \bf m)+\alpha$ odd.
\end{theorem}

Theorem~\ref{MULT}  is the main result of this section, where the dimension and the degree of the variety $X$ are computed.

\begin{theorem}\label{MULT} Adopt Data {\rm\ref{data1}}. Further assume that $d \le n$, that $I$ is a linearly presented Gorenstein ideal of height three, and that
$I$ satisfies $G_d$.
Then 
$A$ has dimension $d$ and multiplicity
$$e(A)=\sum\limits_{i=0}^{\lfloor\frac{n-d}2\rfloor}\binom{n-2-2i}{d-2},$$
which is equal to $$\begin{cases}
\text{the number of monomials of even degree at most $n-d$ in $d-1$ variables}&\text{if $d$ is odd}\\
\text{the number of monomials of odd degree at most $n-d$ in $d-1$  variables}&\text{if $d$ is even.}
\end{cases}$$

\end{theorem}
\begin{proof} We may assume that the field $k$ is infinite. We first observe that the ideal $I$ satisfies the hypotheses of Proposition \ref{xie}; see \ref{Intro-to-BA}.
 According to Proposition \ref{xie}.\ref{3.6.b}, Proposition~\ref{regular}.\ref{3.5.b}, and Proposition \ref{xie}.\ref{3.6.c}, the dimension of $A$ is $d$ and  the multiplicity of $A$ is 
\[e(A)=e\left(\frac{R}{(f_1, \ldots, f_{d-1}):I}\right)\, ,\]
where $f_1, \ldots, f_{d-1}$ are general $k$-linear combinations of homogeneous minimal  generators of $I$. 

Write $J=(f_1, \ldots, f_{d-1}):I$. 
This ideal is a $(d-1)$-residual intersection of $I$ by (\ref{RI}).
Apply Theorem~\ref{Main-C} with
 $s$ replaced by $d-1$ and    $\alpha$ replaced by  $n-d+1$ in order to obtain
$$e(A) = e(R/J)=\sum_{i=0}^{\lfloor\frac {n-d}2\rfloor}\binom{n-2-2i}{d-2}.$$
\end{proof}

\smallskip

\section{Results from ``Degree bounds for local cohomology''}\label{from-DBLC}

An important step in the proof of Theorem~\ref{MainT} is to show that an ideal is of fiber type (see \ref{LT-FT} for a definition) and to identify a power of the maximal ideal that annihilates the torsion of the symmetric algebra. In   (\ref{apr30'})  and \ref{LT-FT} we have seen that this amounts to finding degree bounds for local cohomology modules. Such degree bounds are given in \cite{KPU-DBLC}.  In the present section we recall some of the results from \cite{KPU-DBLC}.
These results  are applied to the study of blowup algebras in Sections~\ref{RealAppToBA} and \ref{uptoradical}.

\begin{data}\label{data5}
Let $R$ be a   
non-negatively graded  Noetherian algebra over a local ring $R_0$ with $\dim  R=d$,  $\mathfrak m$ be the maximal homogeneous ideal of $R$,  
$M$ be a 
 graded 
$R$-module, and \[C_{\bullet}: \qquad \ldots \ \longrightarrow C_1 \longrightarrow C_0 \longrightarrow 0 \]be a homogeneous complex of finitely generated  graded
$R$-modules with $\HH_0(C_{\bullet}) \cong M$.
\end{data}

\smallskip

In  \cite{KPU-DBLC} we find bounds
on  the degrees of interesting elements of the local cohomology modules $\HH^i_\mathfrak m(M)$ in terms of information about the ring $R$ and information that can be read from the complex $C_{\bullet}$.
The ring $R$ need not be a polynomial ring, the complex $C_\bullet$ need not be finite, need not be acyclic, and need not consist of free modules, and the parameter $i$ need not be zero. Instead, we impose hypotheses on the Krull dimension of $\HH_j(C_\bullet)$ and the depth of $C_j$ in order to make various local cohomology modules $\HH_\mathfrak m^\ell(\HH_j(C_\bullet))$  and $\HH_\mathfrak m^\ell(C_j)$ vanish.


\smallskip

Recall that if $M$ and $N$ are modules over a ring, then $N$ is a {\it subquotient} of $M$ if $N$ is isomorphic to a submodule of a homomorphic image of $M$ or, equivalently, if $N$ is isomorphic to a homomorphic image of a submodule of $M$. Also recall the numerical functions of \ref{numerical-functions}.


The next proposition is \cite[\localCoh]{KPU-DBLC}. 
\begin{proposition}\label{localCoh}
Adopt Data~{\rm\ref{data5}}.
Fix an integer $i$ with 
$0\le i\le d$. Assume that
\begin{enumerate}[\rm(1)]
\item\label{localCoh-a} 
$j+i+1\le \depth C_j \ $ for all $j$ with $0\le j\le d-i-1$, and
\item \label{localCoh-b}
$\dim \HH_j(C_{\bullet})\le j+i\ $ for all $j$ with $1\le j$.
\end{enumerate}
Then \begin{enumerate}[\rm(a)]\item\label{localCoh-1}
$\HH^i_{\mathfrak m}(M)$ is a graded subquotient of  $\,  \HH^d_{\mathfrak m}(C_{d-i})$, and
\item \label{localCoh-2} $a_i(M)\le b_0 (C_{d-i})+ a(R)$.  
\end{enumerate}\end{proposition}

\vspace{0.00cm}

\begin{remark}\label{rmk-3} Typically, one applies Proposition~\ref{localCoh} when the modules $C_j$ are maximal Cohen-Macaulay modules (for example, free modules over a Cohen-Macaulay ring) because, in this case, hypothesis (\ref{localCoh-a}) about $\depth C_j$ is automatically satisfied. Similarly, 
if 
\begin{equation}\label{easy-b}\HH_{j}({C_\bullet}_\mathfrak p)=0\  \mbox{ for all $j$ and $\mathfrak p$ with } 
1\le j\le d-i-1,\ 
\mathfrak p\in \operatorname{Spec}(R), \text{ and }
 i+2\le \dim R/\mathfrak p,
\end{equation} then hypothesis (\ref{localCoh-b}) is satisfied.
\end{remark}

\smallskip

We record an immediate consequence of Proposition~{\rm\ref{localCoh}}; this is \cite[\cortolocalCohone]{KPU-DBLC}.
\begin{corollary}\label{cor-to-localCoh-1}
Adopt the hypotheses of Proposition~{\rm\ref{localCoh}}, with $i=0$. 
Then the following statements hold.
\begin{enumerate}[\rm(a)]
\item\label{cor-to-localCoh-1-a} If $C_d=0$, then $\HH^0_\mathfrak m(M)=0$.
\item\label{cor-to-localCoh-1-b}  If $C_d\neq 0$, then 
$$[\HH^0_\mathfrak m(M)]_p=0 \quad \text{for all $p$ with}\quad   \maxgendeg(C_d)+a(R)< p.$$
\end{enumerate}
\end{corollary}

\medskip

The next result, an obvious consequence of Corollary~\ref{cor-to-localCoh-1}, is \cite[\cortolocalCoh]{KPU-DBLC}.
\begin{corollary}\label{cor-to-localCoh}
Adopt the hypotheses of Proposition~{\rm\ref{localCoh}}, with $i=0$. Assume further that $R=k[x_1,\dots,x_d]$ is a standard graded polynomial ring over a field and 
that the subcomplex  $$C_{d}\longrightarrow \ \ldots \ \longrightarrow C_1 \longrightarrow C_0\longrightarrow 0$$ of $C_{\bullet}$ is a $q$-linear complex of free $R$-modules, for some integer $q$ {\rm(}that is, $C_j\cong R(-j-q)^{\beta_j}$ for $0\le j\le d${\rm)}. 
Then $\HH^0_\mathfrak m(M)$ is concentrated in degree $q$; that is, $[\HH^0_\mathfrak m(M)]_p=0 \, $ for all $p$ with $p\neq q$.
\end{corollary}

\smallskip

The next theorem is \cite[\SJDcortwo]{KPU-DBLC} and the main result of this section. It provides an upper bound for the generation degree of the zeroth local cohomology. 
\begin{theorem}\label{SJD-cor2} Let $R=k[x_1,\dots,x_d]$ be a standard graded polynomial ring over a field, with maximal homogeneous ideal $\mathfrak m$, let 
$$C_{\bullet}:\qquad \dots \ \lto C_2 \lto C_1 \lto C_0\lto 0$$ be a homogeneous complex of finitely generated graded $R$-modules, and $M=\HH_0(C_{\bullet})$. Assume that $\dim \HH_j(C_\bullet)\le j$ whenever $1\le j\le d-1$ and that $\min\{d,j+2\}\le \depth C_j$ whenever $0\le j\le d-1$. Then  $$b_0(\HH_{\mathfrak m}^0(M))\le \maxgendeg (C_{d-1})-d+1.$$
\end{theorem}

\smallskip

\begin{remark}\label{rmkAug25} The hypotheses of   Theorem~\ref{SJD-cor2} about depth and dimension are satisfied if the modules $C_j$ are free and condition (\ref{easy-b}) holds.
\end{remark}

\bigskip

\section{Bounding the degrees of defining equations of Rees rings}\label{RealAppToBA}

In this section we apply the local cohomology techniques of \cite{KPU-DBLC}
to bound the degrees of the defining equations of Rees rings.
This style of argument was inspired by work of Eisenbud, Huneke, and Ulrich; see \cite{EHU}.

To apply the results of  \cite{KPU-DBLC}, most notably Corollary~\ref{cor-to-localCoh-1} and Theorem~\ref{SJD-cor2} of the present paper,  we need `approximate resolutions' $C_{\bullet}$ of the symmetric powers of $I(\delta)$, for an ideal $I$ as in Data \ref{data1} that satisfies $G_d$. 
If $I$ is perfect of height two,  we take `$C_{\bullet}$' to be a homogeneous strand of a Koszul complex. The fact
that this choice of `$C_{\bullet}$' satisfies the hypotheses of  Corollary~\ref{cor-to-localCoh-1} and Theorem~\ref{SJD-cor2} 
is shown in the proof of Theorem~\ref{App-to-BA-1}.a. If $I$ is 
Gorenstein of height three, we take `$C_{\bullet}$' to be one of the complexes $\mathcal D^q_{\bullet}(\varphi)$ from \cite{KU-Fam}. In this paper these complexes are introduced  in (\ref{D-twists}); they are shown to 
satisfy the hypotheses of  Corollary~\ref{cor-to-localCoh-1} and Theorem~\ref{SJD-cor2}
 in Lemma~\ref{Acyclic}.

In Theorem~\ref{App-to-BA-1} we carefully record the degrees of the entries in a minimal homogeneous presentation matrix for $I$; but we do not insist that these entries be linear. For the ideal $\mathcal A$ of the symmetric algebra defined in \ref{LT-FT}, we give an explicit exponent $N$ so that $\mathfrak m^N \cdot \mathcal A=0$ and we obtain an explicit bound  for the largest $x$-degree $p$ so that $\mathcal A_{(p,*)}=\bigoplus_{q} \A_{(p,q)}$ contains a minimal bi-homogeneous generator of $\mathcal A$.   

Assertions (\ref{App-1-a-i}) and (\ref{App-1-a-ii})
 of Theorem~\ref{App-to-BA-1}, about perfect height two ideals,  are significant generalizations of Corollaries 2.5(2) and 2.16(1), respectively, in
\cite{KPU-BSSA}, where the same results are obtained for $d=2$ and $n=3$. The starting point for \cite{KPU-BSSA} is the perfect pairing of  Jouanolou \cite{Jo96,Jo97}, which exists because the symmetric algebra of a three generated height two perfect ideal is a complete intersection. With the present hypotheses, the symmetric algebra is only a complete intersection on the punctured spectrum of $R$ and there is no perfect pairing analogous to the pairing of Jouanolou.

Theorem~\ref{App-to-BA-1} is particularly interesting when the ideal $I$ is linearly presented.  Indeed, assertions (\ref{App-1-a-ii}) and (\ref{App-1-b-ii}) say that $I$ is of fiber type. 
For  height two perfect ideals this was already observed in \cite{MU}. For height three Gorenstein ideals, this is a new result and a main ingredient in the proof of Theorem~\ref{MainT}. The application of Theorem~\ref{App-to-BA-1}  to the situation where $I$ is a linearly presented height three Gorenstein ideal is recorded as Corollary~\ref{fiber type}.

\begin{theorem}\label{App-to-BA-1}Adopt Data~{\rm\ref{data1}} and recall the bi-homogeneous ideal $\A$ of the symmetric algebra defined in \ref{LT-FT}. Assume further that 
$I$ satisfies the condition $G_d$. 
\begin{enumerate}[\rm(a)]
\item \label{App-1-a}Assume that the ideal $I$ is perfect of height two and the column degrees of a homogeneous Hilbert-Burch matrix minimally presenting $I$ are $\epsilon_1\ge \epsilon_2\ge \dots\ge \epsilon_{n-1}$. Then the following statements hold when $d \leq n-1$.
\begin{enumerate}[\rm(i)]
\item\label{App-1-a-i} 
If \,$\sum_{i=1}^d (\epsilon_i-1)<p$, then $\mathcal A_{(p,*)}=0$. In particular,  $$\mathfrak m^{1+\sum_{i=1}^d (\epsilon_i-1)}\mathcal A=0.$$  
\item\label{App-1-a-ii} If  $\mathcal A_{(p,*)}$ contains a minimal bi-homogeneous generator of $\mathcal A$, then $p\le \sum_{i=1}^{d-1}(\epsilon_i-1)$.
\end{enumerate}
\item\label{App-1-b} Assume that  the ideal $I$ is Gorenstein of height three  and every entry 
 of a homogeneous alternating matrix minimally presenting $I$ has degree $D$.  Then the following statements hold.
\begin{enumerate}[\rm(i)] 
\item\label{App-1-b-i} If 
$$\begin{cases}
d(D-1)<p&\text{when $d$ is odd }\\
(d-1)(D-1)+\frac{n-d+1}2D-1<p&\text{when $d$ is even\,,}
\end{cases}$$
 then $\mathcal A_{(p,*)}=0$. 
In particular, 
$$\begin{cases}
\text{$\mathfrak m^{d(D-1)+1}\mathcal A=0$}&\text{if $d$ is odd}\\ 
\text{$\mathfrak m^{(d-1)(D-1)+\frac{n-d+1}2D}\mathcal A=0$}&\text{if $d$ is even\,.}\end{cases}$$

\item\label{App-1-b-ii} If  $\mathcal A_{(p,*)}$ contains a minimal bi-homogeneous generator of $\mathcal A$, then $p\le (d-1)(D-1) \,$.
\end{enumerate}
\end{enumerate}
\end{theorem} 

\vspace{0.0cm}

\begin{Remarks}\begin{enumerate}[\rm(a)]\item We may safely assume that $d \leq n-1$ in Theorem \ref{App-to-BA-1} because otherwise $\mathcal A = 0 \, ;$
see (\ref{Aug232016}). \item The assumptions on the degrees of the entries in a presentation matrix do not impose any restriction: In (a), the inequalities  $\epsilon_1\ge \epsilon_2\ge \dots\ge \epsilon_{n-1}$ can be achieved by a permutation of the columns. In (b), all entries necessarily have the same degree, as can be seen from the symmetry of the minimal homogeneous $R$-resolution of $R/I$.  
\item Results similar to, but more complicated than, Theorem~\ref{App-to-BA-1} can be obtained without the assumption that the ideal $I$ is generated by forms of the same degree.
\item The degree $\delta$ of the forms $g_1, \ldots, g_n$ in Data~\ref{data1} is $\sum_{i=1}^{n-1}\epsilon _i$ in the setting of Theorem \ref{App-to-BA-1}.\ref{App-1-a} and is
$\frac{n-1}{2} \, D\, $ in Theorem \ref{App-to-BA-1}.\ref{App-1-b}; see \cite{B} and \cite{BE}, respectively.

\end{enumerate}\end{Remarks}

\vspace{0.0cm}

\begin{proof} 
We have seen in (\ref{apr30'}) that the ideal 
$\mathcal A$ of $\Sym(I(\delta))$ is equal to the local cohomology module $\HH_{\mathfrak m}^0(\Sym(I(\delta)))$.
We apply 
Corollary~\ref{cor-to-localCoh-1} and Theorem~\ref{SJD-cor2} to learn about the $R$-modules
\begin{equation}\label{the point}\HH_\mathfrak m^0(\Sym_q(I(\delta)))= \bigoplus_{p}\mathcal A_{(p,q)}=\mathcal A_{(*,q)}.\end{equation}

\medskip\noindent (\ref{App-1-a}) 
Let 
$\varphi$ be a 
minimal 
homogeneous 
Hilbert-Burch matrix for $I$. In other words, $\varphi$ is 
an 
$n\times (n-1)$ matrix 
with homogeneous entries from $R$ such that
$$0\lto F_1 
\stackrel{\varphi}\longrightarrow F_0
\stackrel{\underline g}\longrightarrow
I(\delta)\lto 0$$
is a homogeneous 
exact sequence of $R$-modules, where $\underline g$ is the row vector $[g_1,\dots,g_n]$ of \ref{def1}, $$F_0=R^n \quad \text{and} \quad F_1=\bigoplus\limits_{i=1}^{n-1} R(-\epsilon_i).$$ 
It induces a bi-homogeneous presentation of the symmetric algebra
\[(\Sym(F_0) \otimes_R F_1)(0,-1) \stackrel{\underline {\ell}}\lto\Sym(F_0)\lto \Sym(I(\delta))\lto 0\, ,
\]
where $\Sym(F_0)$ is the standard bi-graded polynomial ring $S$ of \ref{grading} and $\underline{\ell}$ is the row vector $[T_1,\dots, T_n]~\cdot~\varphi$ of \ref{def1}. We consider the Koszul complex of $\underline{\ell}$,
\[\mathbb K_{\bullet}= \bigwedge^{\bullet}_S( (S\otimes_R F_1)(0,-1))
= S\otimes_R \bigwedge^{\bullet}_R (F_1(0,-1))\, .
\]
It is a bi-homogeneous complex of free $S$-modules with $\HH_0(\mathbb K_{\bullet})$ equal to $\Sym(I(\delta))$. 

The ideal  $I$ is perfect of height two and  satisfies $G_d$ (see \ref{Gd-def}). It follows that 
the sequence $\underline{\ell}=\ell_1, \ldots, \ell_{n-1}$ is a regular sequence 
locally on the punctured spectrum of $R$ (see Observation~\ref{triv});  and therefore the Koszul complex $\mathbb K_{\bullet}$ is 
 acyclic on the punctured spectrum of $R$. 
In particular, for each non-negative integer $q$,  the component of degree $(\star, q)$ of $\mathbb K_{\bullet}$ is a homogeneous complex of graded free $R$-modules,  has zeroth homology equal to $\Sym_q(I(
\delta))$, and is exact on the punctured spectrum of $R$. 
This component  looks like
$$C_\bullet^q:\quad 0\lto C_{n-1}^q\lto C_{n-2}^q\lto \ \ldots \ \lto C_{1}^q\lto C_{0}^q \lto 0\,,$$ with 
$$C_r^q\, =\, \Sym_{q-r}(F_0)\otimes_R \bigwedge^{r} F_1 \, \cong
 \bigoplus_{1\le i_1<\dots<i_{r}\le n-1} R^{b_{r}}(-\sum\limits_{j=1}^r \epsilon_{i_j})$$
 for 
 $$ b_r=\rank (\Sym_{q-r}(F_0))=\binom{q-r+n-1}{n-1}\, .$$ 
The assumption $\epsilon_1\ge \epsilon_2\ge \dots\ge \epsilon_{n-1}$ yields that if $C^q_r\neq 0$ then 
\begin{equation}\label{mgd}\maxgendeg(C^q_r)=\sum\limits_{i=1}^r \epsilon_i\quad\text{for $1\le r\le n-1$}.\end{equation}
Notice that if $C^q_r=0$ then $\maxgendeg(C^q_r)= - \infty \, .$

Apply  Corollary~\ref{cor-to-localCoh-1}  
(see also Remark~\ref{rmk-3}) to the complex $C_\bullet^q$ to conclude that  
\begin{equation}\label{zzz}[\HH^0_\mathfrak m(\Sym_q(I(\delta)))]_p=0 \quad\text {for all $p$ with $\maxgendeg(C_d^q)+a(R)<p$.}\end{equation} 
Recall that the index  $q$ in (\ref{zzz}) is  arbitrary. Apply (\ref{the point}), (\ref{mgd}), and (\ref{zzz}) to conclude that
$$\mathcal A_{(p,*)}=0 \quad\text {for all $p$ with $\sum\limits_{i=1}^d \epsilon_i-d<p$.}$$ 
Assertion (\ref{App-1-a-i}) has been established.

Apply  
Theorem~\ref{SJD-cor2}  (see also Remark~\ref{rmkAug25})
to the complex $C_\bullet^q$ 
to conclude that 
every minimal homogeneous generator $\alpha$ of $\HH_\mathfrak m^0(\Sym_q(I(\delta)))$ satisfies 
$$\deg\alpha\le \maxgendeg (C_{d-1}^q)-d+1.$$
Once again,  (\ref{the point}) and (\ref{mgd}) yield that 
$$\maxgendeg(\mathcal A_{(*,q)})\le \sum\limits_{i=1}^{d-1}\epsilon_i-d+1.$$
The parameter $q$ remains arbitrary. Every minimal bi-homogeneous generator of $\mathcal A$ is a minimal homogeneous generator of $\mathcal A_{(*,q)}$, for some $q$. We conclude that if $\mathcal A_{(p,q)}$ contains a minimal bi-homogeneous generator of $\mathcal A$, then $p\le \sum\limits_{i=1}^{d-1}\epsilon_i-d+1$, and this establishes assertion 
(\ref{App-1-a-ii}). 

\medskip\noindent(\ref{App-1-b}) 
Let $\varphi$ be an  $n\times n$ homogeneous alternating matrix which  minimally presents $I$, and, for each non-negative integer $q$, let 
$$\mathcal D^q_{\bullet}(\varphi):
\quad 0\lto \mathcal D^q_{n-1}\lto \mathcal D^q_{n-2}\lto \ \dots \ \lto \mathcal D^q_{1}\lto \mathcal D^q_{0}\lto 0$$ be the complex from \cite[Def.~2.15 and Fig.~4.7]{KU-Fam} that is associated to the alternating matrix $\varphi$. The zeroth homology of $\mathcal D^q_{\bullet}(\varphi)$ is $\Sym_q(I(\delta))$; see  \cite[4.13.b]{KU-Fam}.
The hypothesis that every entry of $\varphi$ has degree $D$ yields that
\begin{equation}\label{D-twists}\mathcal D^q_r=\begin{cases} K_{q-r,r}=R(-rD)^{\beta^q_r}&\text{if $r\le \min\{q, n-1\}$}\\
Q_q=R(-(r-1)D-\frac{(n-r+1)}2D)&\text{if $r=q+1$, $q$ is odd, and $r\le n-1$}\\
0&\text{if $r=q+1$ and $q$ is even}\\
0&\text{if $\min\{q+2,n\}\le r$}\\
\end{cases}\end{equation}for some non-zero Betti numbers $\beta^q_r$. Indeed, the graded $R$-module $\Sym_q(I(\delta))$ is generated in degree zero; moreover, each map $K_{a,b}\to K_{a+1,b-1}$ is linear in the entries of $\varphi$ (see \cite[2.15.d]{KU-Fam}) and the map ${Q_q\to K_{0,q}}$ may be represented by a column vector whose entries consist of the Pfaffians of the principal ${(n-r+1)\times (n-r+1)}$ alternating submatrices of $\varphi$ (see \cite[2.15.c and 2.15.f]{KU-Fam}).

 Notice from (\ref{D-twists}) that 
\begin{equation}\label{if} \text{if $\mathcal D^q_r=Q_q$, then $\mathcal D^q_{r+1}=0 \, .$}\end{equation} 
Notice also that 
\begin{align}
\max\limits_{\{q\mid \mathcal D_d^q\neq 0\}}\maxgendeg (\mathcal D_d^q)&
\leq \begin{cases}
dD&\text{if $d$ is odd}\\
(d-1)D+\frac{n-d+1}2D&\text{if $d$ is even\,,}\label{5.7.5}
\end{cases}
\\ 
\max\limits_{\{q\mid \mathcal D_d^q\neq 0\}}\maxgendeg (\mathcal D_{d-1}^q)& \leq (d-1)D \, ;\label{5.7.6} 
\end{align}
one uses (\ref{if}) in order to see (\ref{5.7.6}) because the condition $\mathcal D_d^q\neq 0$ forces $\mathcal D_{d-1}^q=K_{q-d+1,d-1}$.

By hypothesis the ideal $I$ satisfies the condition $G_d$; so $I$ satisfies the hypothesis (\ref{WG}) of Lemma~\ref{Acyclic} for $s=d-1$; and therefore, the complex $(\mathcal D^q_{\bullet}(\varphi))_{\mathfrak p}$ is acyclic for every $\mathfrak p\in \operatorname{Spec}(R) $ with $\dim R_{\mathfrak p}\le d-1$ and every non-negative integer $q$. The hypotheses of 
Corollary~\ref{cor-to-localCoh-1} and Theorem~\ref{SJD-cor2} 
are satisfied by the complexes $\mathcal D^q_{\bullet}(\varphi)$. 
We deduce that 
either 
$$\text{$\mathcal D^q_d$ and $\HH^0_\mathfrak m(\Sym_q(I(\delta)))$ are both zero;}$$
or else $\mathcal D^q_d\neq 0$,  
$$[\HH^0_\mathfrak m(\Sym_q(I(\delta)))]_{p}= 0\ \text{ \ for all $p$ with $\,\maxgendeg (\mathcal D^q_d)+a(R)< p $},$$
and every 
minimal homogeneous generator $\alpha$ of $\HH^0_\mathfrak m(\Sym_q(I(\delta)))$ satisfies 
$$\deg \alpha \le \maxgendeg(\mathcal D^q_{d-1})-d+1.$$
In the case $\mathcal D^q_d\neq 0$ we apply (\ref{the point}), (\ref{5.7.5}), and (\ref{5.7.6}) to conclude that
$$\mathcal A_{(p,q)}=0\quad \text{ for } \begin{cases} dD-d<p &\text{if $d$ is odd}\\
(d-1)D+\frac{n-d+1}2D-d< p &\text{if $d$ is even\,,}\end{cases}$$
and every 
minimal homogeneous generator $\alpha$ of $\mathcal A_{(*,q)}$ satisfies 
$$\deg \alpha \le (d-1)D-d+1.$$
Keep in mind that if $\alpha\in \mathcal A_{(p,q)}$ is a minimal bi-homogeneous generator of $\mathcal A$, then $\alpha$ is also a minimal homogeneous generator of $\mathcal A_{(*,q)}$. Assertions (\ref{App-1-b-i}) and (\ref{App-1-b-ii}) have been established. 
\end{proof}

\vspace{0.0cm}

Two small facts were used in the proof of Theorem~\ref{App-to-BA-1} to guarantee  that the complexes $C_{\bullet}^q$ and $\mathcal D^q_{\bullet}(\varphi)$ satisfy the hypotheses of Corollary~\ref{cor-to-localCoh-1} and Theorem~\ref{SJD-cor2}. We prove these facts now.
\begin{observation}\label{triv}
Let $R$ be a 
Cohen-Macaulay local ring and 
$I$ be an ideal of positive height. 
Assume that $\mu(I_{\mathfrak p})\le \operatorname{dim}  R_{\mathfrak p}+1
\mbox{\ for each prime ideal \ } \mathfrak p \in V(I)$. If
$\Sym(I) \cong P/J$ is any homogeneous presentation, where $P$ is a standard graded polynomial ring over $R$ in $n$ variables and $J$ is a homogeneous ideal of $P$, then $J$ has height $n-1$. 
\end{observation}
\begin{proof} 
One has
$$\htt \, J = \dim P - \dim P/J =\dim R+n-\dim \, \Sym(I) = n-1,$$
where the first equality holds because $R$ is Cohen-Macaulay local and $J$ is homogeneous and the last equality follows from the dimension formula for symmetric algebras; see \cite[2.6]{HR}. 
\end{proof}

\begin{lemma}\label{Acyclic} Let $R$ be a Cohen-Macaulay ring and $I$ be a perfect Gorenstein ideal of height three. Let $\varphi$  be an alternating  presentation matrix for $I$, $\{\mathcal D^q_{\bullet}(\varphi)\}$ be the family of complexes in \cite{KU-Fam} which is associated to $\varphi$, and $s$ be an integer. If 
\begin{equation}\label{WG}
\mu(I_{\mathfrak p})\le \operatorname{dim}  R_{\mathfrak p}+1
\mbox{\quad for each prime ideal \ } \mathfrak p \in V(I)  \mbox{\  with \ } \operatorname{dim}  R_{\mathfrak p} \le s\, ,
\end{equation} then the complex $[\D^{q}_{\bullet}(\varphi)]_{\mathfrak p}$ is acyclic for every  $\mathfrak p\in \operatorname{Spec}  (R)$  with $\operatorname{dim}  R_{\mathfrak p} \le s$ and every non-negative integer $q$. 
\end{lemma}

\begin{proof}   Let $\mathfrak p\in \operatorname{Spec}  (R)$  with $\operatorname{dim}  R_{\mathfrak p} \le s$. Notice that $\D^{q}_{\bullet}(\varphi)_{\mathfrak p}=\D^{q}_{\bullet}(\varphi_{\mathfrak p})$. This complex has length at most $n-1$, where $n$ is the size of the matrix $\varphi_{\mathfrak p}$ according to  \cite[4.3.d]{KU-Fam}; see also (\ref{D-twists}).  Moreover, by  assumption  (\ref{WG}), the matrix $\varphi_{\mathfrak p}$ satisfies  condition ${\rm WMC}_1$, and therefore ${\rm WPC}_1$, see \cite[5.8, 5.9, 5.11.a]{KU-Fam}. Now  
\cite[6.2]{KU-Fam} shows that the complex $\D^{q}_{\bullet}(\varphi_{\mathfrak p})$ is acyclic.
\end{proof}

Theorem~\ref{App-to-BA-1} is particularly interesting when $I$ is a linearly presented height three Gorenstein ideal. One quickly deduces that such ideals are of fiber type; see \ref{LT-FT} for the definition. This is a new result and a main ingredient in the proof of Theorem~\ref{MainT}.

\begin{corollary}\label{fiber type} Adopt the setting of Theorem~{\rm\ref{App-to-BA-1}.\ref{App-1-b}} with $D=1$. Then the ideal $I$ is of fiber type. Furthermore, $\mathcal A$ is annihilated by $\mathfrak m$ when $d$ is odd and by $\mathfrak m^{\frac{n-d+1}2}$ when $d$ is even.\end{corollary}

\begin{proof}
Theorem~\ref{App-to-BA-1}.\ref{App-1-b-ii} shows that the ideal $\mathcal A$ of $\Sym(I(\delta))$ is generated by $\mathcal A_{(0,\star)}$\,, which means that $I$ is of fiber type; see \ref{LT-FT}. 
The assertion about which power of $\mathfrak m$ annihilates $\mathcal A$ is explicitly stated in Theorem~\ref{App-to-BA-1}.bi.  
\end{proof}

\bigskip
\section{The defining ideal of Rees rings up to radical}\label{uptoradical}

In this section we mainly work under the assumption that the ideal $I$ of Data~\ref{data1} is of linear type on the punctured spectrum of $R$ and that a sufficiently high symmetric power $\Sym_t(I(\delta))$ has an approximate free resolution that is linear for the first $d$ steps. With these hypotheses alone we prove, somewhat surprisingly, that the defining ideal $\J$ of the Rees ring has the expected form, up to radical; see \ref{def1} for the definition of expected form. We also describe the defining ideal $I(X)$ of the variety $X$ up to radical and we deduce that $I_d(B')$ 
has maximal possible height, where $B'$ is the Jacobian dual of the linear part of a minimal homogeneous presentation matrix of $I$; see \ref{def1} for the definition of Jacobian dual. This is done in Theorem~\ref{correct grade}. 
The computation of the height of $I_d(B')$ is relevant for the proof of Theorem~\ref{MainT} because it ensures that the complexes constructed in  \cite{KPU-ann} are free resolutions of $I_d(B)$ and of $I_d(B)+C(\varphi)$, where $B=B'$ is the full Jacobian dual in this case and $C(\varphi)$ is the content ideal of Definition \ref{C-of-phi}.
A step towards the proof of 
Theorem~\ref{correct grade} is  Lemma~\ref{Y1.6+}, where we show that up to radical, $\A$ is the socle of the symmetric algebra as an $R$-module and $I$ is of fiber type; see \ref{LT-FT} for the definition of fiber type. We deduce this statement from the other parts of the same lemma, which say that even locally on the punctured spectrum of the polynomial ring $S$, the ideal $\A$ belongs to the socle and $I$ is of fiber type. These facts are proved using the methods of \cite{KPU-DBLC}, most notably Corollary~\ref{cor-to-localCoh} of the present paper. 


The assumption that $I$ is of linear type on the punctured spectrum of $R$ is rather natural; see \ref{Intro-to-BA}. It is satisfied for instance if $I$ is perfect of height two or 
Gorenstein of height three and  $I$ satisfies $G_d$.
If such an ideal is, in addition,  linearly presented, then the hypothesis about an approximate resolution $C_{\bullet}$ of $\Sym_t(I(\delta))$ holds as well;
see Remark~\ref{ass-satisfied}. 

The index $t$ of the symmetric power $\Sym_t(I(\delta))$
in Lemma~\ref{Y1.6+} and Theorem~\ref{correct grade} 
can be any integer greater than or equal to the relation type of $I$. 
The {\it relation type} of an ideal  $I$ as in Data~\ref{data1} is the largest integer $q$ such that  $\mathcal A_{(*,q)}$ contains a minimal bi-homogeneous  generator of $\mathcal A$. 
We recall the bi-homogeneous ideal $\, \A=\J/\L\, $ of $\Sym(I(\delta))$ as well as the  bi-homogeneous ideals $\L$ and $\J$ of $S$ that define the symmetric algebra and the Rees ring, respectively; see \ref{grading} and \ref{LT-FT}.

\begin{lemma}\label{Y1.6+}Adopt Data~{\rm\ref{data1}}.
 Suppose further that $I$ is of linear type on the punctured spectrum of $R$ and that  for some index $t$ greater than or equal to the relation type of $\, I$ 
there exists a homogeneous complex 
\[C_{\bullet}: \qquad \ldots \,  \longrightarrow C_1 \longrightarrow C_0 \longrightarrow 0 \]
 of finitely generated graded free $R$-modules such that 
\begin{enumerate}[\rm (1)]\item $\HH_0(C_{\bullet})=\Sym_t(I(\delta))\, ,$
\item the subcomplex
$\, C_d\to\ldots \to C_0\, $ of
$C_{\bullet}$ is linear {\rm(}that is, $C_i=R(-i)^{\beta_i}\, $ for $0\le i\le d${\rm)}, and
\item $\dim\HH_j(C_{\bullet})\le j\, $ \,for $1\le j\le d-1$.\end{enumerate}
Let $\mathfrak M$ be the ideal $(x_1,\dots,x_d,T_1,\dots,T_n)$ of $S$. Then the following statements hold.\begin{enumerate}[\rm(a)]\item\label{Y2.8.1.i} On   $\Spec (S)\setminus \{\mathfrak M\}$, the ideal $\mathcal A$ of $\, \Sym(I(\delta))$ is annihilated by $\mathfrak m$; that is, $$(\m \A)_\P=0\quad \text{ for every}\ \ \P \in \Spec(S)\setminus \{\M\}\,.$$
\item\label{Y2.8.1.ii} On $\Spec (S)\setminus \{\mathfrak M\}$, $I$ is of fiber type; that is, $$\A_{\P}=I(X)\cdot\Sym(I(\delta))_{\P}\,   \text{\quad for every} \ \ \P \in \Spec(S)\setminus \{\M\}\,.$$
\item\label{Y2.8.1.iii} The ideal $\J$ of $\, S$  satisfies
$$\J=\sqrt{\L:_S \m}=\sqrt{(\L,I(X))}\,.$$
\end{enumerate}
\end{lemma}

\begin{proof} 
Recall that $\mathcal A=\HH^0_\mathfrak m(\Sym(I(\delta)))$ since $I$ is of linear type on the  punctured spectrum of $R$; see (\ref{A=H}). In particular, 
$$\mathcal A_{\,(\star,t)}=\HH^0_\mathfrak m(\Sym_t(I(\delta)))\, .$$
Apply Corollary~\ref{cor-to-localCoh} to the $R$-module $\Sym_t(I(\delta))$ to conclude that
$$\mathcal A_{\, (\star, t)}=\mathcal A_{\,(0,t)}\, .
$$
Therefore 
$$\mathcal A_{\,(\star,\ge t)}=\mathcal A_{\,(0,\ge t)}
$$
because, by the definition of relation type,  $\mathcal A_{\,(\star, t)}$ generates $\mathcal A_{\,(\star, \ge t)}$ as a module over $\Sym(I(\delta))$, hence over 
$\Sym(I(\delta))_{(0,\star)}$. 
It follows that $(T_1,\ldots,T_n)^t \A\subset \mathcal A_{\,(0,\ge t)}$\,, which gives
$$(T_1,\ldots,T_n)^t \A_{\,(>0, \star) }=0\, .
$$
Since, moreover,  $\A_{\,(>0, \star)}$ vanishes locally on $\Spec(R) \setminus \{\m\}$, we deduce that
$$\A_{\,(>0, \star)} =0 \mbox{\ \, locally on\,} \Spec(S) \setminus \{\M\}\, .
$$

This completes the proof of (\ref{Y2.8.1.i}) because $\m \A \subset \A_{\,(>0, \star)}\,$,  and of (\ref{Y2.8.1.ii}) because   $\A=I(X)+\A_{\,(>0, \star)}$.
As for (\ref{Y2.8.1.iii}), we recall the obvious inclusion $\L :_S \m \subset \J$. Part (\ref{Y2.8.1.i}) shows that if $\P \in \Spec(S)\setminus \{\M\}$ contains 
$\L :_S \m $ then $\P$ contains $\J $. Since $\M$ contains $\J$ anyway, we deduce that $\J\subset \sqrt{\L:_S \m}$. The same argument, with part (\ref{Y2.8.1.i}) replaced by (\ref{Y2.8.1.ii}), shows that $\J=\sqrt{(\L,I(X))}\,.$
\end{proof}

\begin{theorem}\label{correct grade}Retain all of the notation and hypotheses of Lemma~{\rm\ref{Y1.6+}}. 
Let $\varphi$ be a minimal homogeneous presentation matrix of $g_1, \ldots, g_n\, ,$ $\varphi '$ be the submatrix of
$\varphi$ that consists of the linear columns of $\varphi$, $B$ be 
a homogeneous Jacobian dual of $\varphi$, and $B'$ be the Jacobian dual of $\varphi '$ in the sense of {\,\rm\ref{def1}}.
Then the following statements hold.
\begin{enumerate}[\rm(a)] 
\item\label{Y2.8.2} The ideal $I(X)$ of $T$ which defines the variety $X$ satisfies
$$I(X)=\sqrt{I_d(B')} \, .$$
\item\label{Y2.8.2.5}  $\operatorname{ht}(I_d(B'))=n-\dim A\, .$
\item\label{Y2.8.3} The ideal $\J$ of $S$ which  defines the Rees ring $\mathcal R(I)$ satisfies
$$\J=\sqrt{(\mathcal L,I_d(B'))}=\sqrt{(\mathcal L,I_d(B))}\,.$$
\end{enumerate}
\end{theorem}

\begin{proof} Part (\ref{Y2.8.2.5}) is an immediate consequence of (\ref{Y2.8.2}). We prove items (\ref{Y2.8.2}) and (\ref{Y2.8.3}) simultaneously. We first notice that $B'$ is a submatrix of $B$, the ideal $I_d(B)$ of $S$ is bi-homogeneous, the two ideals $(\m,I_d(B))$ and $(\m, I_d(B'))$ of $S$ are equal, and the two ideals $I_d(B)\cap T$ and $I_d(B')$ of $T$ are equal.

From Lemma~\ref{Y1.6+}.\ref{Y2.8.1.iii} we know that $$\J=\sqrt{\L:_S\m}\, .$$
The ideal $\mathcal L:_S\mathfrak m$
is the annihilator of the $S$-module $M=\mathfrak m S/\mathcal L$. The ideal $\mathfrak m S$ is generated by  the entries  of $\underline{x}$ and the ideal $\mathcal L$ is generated by the entries of $\underline{x} \cdot B$. It follows that $M$ is presented by 
$\left[\begin{array}{c | c}
\Pi & B  \\
\end{array}
\right],$ where $\Pi$ is a presentation matrix of $\underline{x}$ with entries in $\m$. 
Therefore 
$$\mathcal L:_S\mathfrak m = \ann_S \, M \subset \sqrt{\operatorname{Fitt}_0(M)} \subset \sqrt{(\mathfrak m, I_d(B))}\, .
$$
It follows that 
$$\J \subset \sqrt{(\mathfrak m, I_d(B))}=\sqrt{(\mathfrak m, I_d(B'))}\, .
$$
Intersecting with $T$ we obtain 
$$I(X)=\J \cap T \subset \sqrt{(\mathfrak m, I_d(B'))} \cap T= \sqrt{(\mathfrak m, I_d(B')) \cap T}=\sqrt{ I_d(B')}\, ,
$$
where the last two radicals are taken in the ring $T$. This proves part (\ref{Y2.8.2}). 

From Lemma~\ref{Y1.6+}.\ref{Y2.8.1.iii} we also know that
$$\J=\sqrt{(\L,I(X))}\, ,$$
and then by part  (\ref{Y2.8.2}) 
$$\J=\sqrt{(\mathcal L,I_d(B'))}\,.$$
Finally, recall the inclusions $I_d(B') \subset I_d(B) \subset \J$, and the proof of part (\ref{Y2.8.3}) is complete. 
\end{proof}

\vspace{-0.15cm}
 
 \begin{remark} Retain the notation and the hypotheses of Theorem~\ref{correct grade}, and recall the definition of analytic spread from \ref{graph}. If $d < n$ then $\ell (I) <n \,$;
 and if $\ell(I)<n$ then $\varphi'$ has at least $n-1$ columns. 
 \end{remark}
 \begin{proof}  Write $s$ for the number of columns of $\varphi'$, which is also the number of columns of $B'$.  One has
  $$n-d \leq n-\ell(I)=n-\dim\, A= \htt I_d(B') \le \max \{s-d+1,0\}\, ,$$
where the first inequality holds by (\ref{ell}), the first equality follows from \ref{graph}, the second equality is Theorem~\ref{correct grade}.\ref{Y2.8.2.5}, and the last inequality holds by the Eagon-Northcott bound on the height of determinantal ideals. 
These inequalities show that if $0<n-d$ then $0<n-\ell (I)$ and if $0<n-\ell(I)$ then $n-d \leq s-d+1$.
 \end{proof}

 \begin{remark}\label{ass-satisfied} Adopt Data~\ref{data1}.   Further assume that $I$ is perfect of height two or Gorenstein of height three, $I$ is linearly presented,  and  $I$  satisfies $G_d$. Then the hypotheses of Lemma~\ref{Y1.6+} and Theorem~\ref{correct grade} are satisfied. 
 \end{remark}
 \begin{proof} We know from (\ref{A=H}) and (\ref{apr30'}) that the ideal $I$ is of linear type on the punctured spectrum. 
 
 We now establish the existence of a complex $C_{\bullet}$. In the first case, when  $I$ is perfect of height two, we can take $C_{\bullet}$ to be the complex $C_{\bullet}^t=\mathbb K_{\bullet(\star,t)}$ defined in the proof of Theorem~\ref{App-to-BA-1}.\ref{App-1-a}. This is a linear complex of finitely generated graded free $R$-modules which is acyclic on the punctured spectrum and has $\Sym_t(I(\delta))$ as zeroth homology. These facts are shown in the proof of Theorem~\ref{App-to-BA-1}.a, most notably in (\ref{mgd}).
 
 In the second case, when $I$ is Gorenstein of height three, we  take $C_{\bullet}$ to be the complex $\mathcal D^t_{\bullet}(\varphi)$ used in the proof of Theorem~\ref{App-to-BA-1}.\ref{App-1-b}, with the additional restriction that $d\le t$. This is a  complex of finitely generated graded free $R$-modules which has $\Sym_t(I(\delta))$ as zeroth homology. Moreover, this complex is acyclic on the punctured spectrum, see Lemma~\ref{Acyclic},
 and it is linear for the first $d$ steps, see (\ref{D-twists}). 
 \end{proof}

We end this section by recording a fact that is contained in the proof of Theorem~\ref{App-to-BA-1}.\ref{App-1-b}, but may well be of independent interest. Recall the functions $a_i$ and $\reg$ from \ref{numerical-functions} and \ref{aiM}. 
\begin{proposition}\label{symmetric} 
Adopt Data {\rm\ref{data1}}. Further assume that  $I$ is a linearly presented Gorenstein ideal of height three.
Let $q$ be a non-negative integer. If $\mu(I_{\mathfrak p})\le \dim  R_{\mathfrak p}+1$
  for each prime ideal  $ \mathfrak p\in V(I)$   with $\dim R_{\mathfrak p} \le d-2$,  then 
 $$a_i(\Sym_q(I))\le \begin{cases}  q\delta-i &\text{if $q$ is even or $q\neq d-i-1$}\\
q\delta+\frac{n-q}{2}-1-i& \text{if $q$ is odd and $q =d-i-1$}.\end{cases} 
 $$
 In particular, 
$$\reg \, \Sym_q(I)  \ \begin{cases}=q\delta &\text{if $q$ is even or $d\le q$}\\
\le q\delta+\frac{n-q}{2}-1  & \text{if $q$ is odd and $q \le d-1$}.\end{cases}  $$
\end{proposition}
\begin{proof} Apply Proposition~\ref{localCoh} with 
$M= \Sym_q(I(\delta))$ and $C_{\bullet}$ the complex $\mathcal D^{q}_{\bullet}(\varphi)$ of (\ref{D-twists}).  
Lemma \ref{Acyclic} guarantees  that $\dim \HH_j(\mathcal D^q_{\bullet}(\varphi))\le 1$ for every $j$ with $1\le j$. Thus Proposition~\ref{localCoh}.\ref{localCoh-1} (see also Remark~\ref{rmk-3}) gives that $\HH^i_{\mathfrak m}(\Sym_q(I(\delta)))$ is a graded subquotient of $\HH^d_{\mathfrak m}(\mathcal D^q_{d-i})$ for $0 \le i \le d$. From  (\ref{D-twists})  we see that either $\mathcal D^q_{d-i}=0$ or else
$$ \mathcal D^q_{d-i} \ \cong \begin{cases}   R(
 -d+i)^{\beta^q_{d-i}} &\text{if $q$ is even or $q\neq d-i-1$}\\
R (-d+i+1-\frac{n-q}{2}) & \text{if $q$ is odd and $q =d-i-1$}.\end{cases}  $$

\smallskip

\noindent
Therefore the top degree of  $\HH^d_{\mathfrak m}(\mathcal D^q_{d-i})$ is at most 
$$  \begin{cases}  d-i-d=  -i &\text{if $q$ is even or $q\neq d-i-1$}\\
d-i-1+\frac{n-q}{2}-d= -i-1+\frac{n-q}{2}& \text{if $q$ is odd and $q =d-i-1$}.\end{cases}  $$

\smallskip
\noindent 
Notice the second case cannot occur if $q$ is even or $d\le q$.  

On the other hand, the fact that the $R$-module $\Sym_q(I(\delta)) $ is generated in degree $0$ implies that the regularity of $ \Sym_q(I(\delta))$ is non-negative. Now the assertions follow. 
 \end{proof}

\medskip

\section{The candidate ideal $C(\varphi)$}


\smallskip

In the present section we exhibit the candidate for the defining ideal of the variety $X$.  Most of this material is taken from \cite{KPU-ann}. 
\begin{definition}\label{C-of-phi} Let $k$ be a field and  $n$ and $d$ be positive integers.  
For each $n\times n$ alternating matrix $\varphi$ with linear entries from the polynomial ring $R=k[x_1,\dots,x_d]$,  define 
an ideal $C(\varphi)$ in  the polynomial ring $T=k[T_1,\dots,T_n]$ as follows. Let $B$ be the $d\times n$ matrix with linear entries from $T$ such that  the matrix equation
\begin{equation}\label{FE}[T_1,\dots,T_n]\cdot \varphi=[x_1,\dots,x_d]\cdot B\end{equation} 
holds.  Consider the $(n+d)\times(n+d)$ alternating matrix $$\goth B=\bmatrix \varphi&-B^{\rm t}\\B&0\endbmatrix,$$with entries in the polynomial ring $S=k[x_1,\dots,x_d,T_1,\dots,T_n]$. Let $F_i$ be $(-1)^{n+d-i}$ times the Pfaffian of $\goth B$ with row and column $i$ removed. 
View  each $F_{i}$ as a polynomial in $T[x_1,\dots,x_d]$ and  
let $$C(\varphi)=c_{T}(F_{n+d})$$ be
the content ideal of $F_{n+d}$ in $T$. \end{definition}


A description of $C(\varphi)$ which is almost coordinate-free is given in \cite{KPU-ann}. The description depends on the choice of a direct sum decomposition of the degree one component of the ring $T$ into $V_1\oplus V_2$ where $V_1$ has dimension one. (The almost coordinate-free description of $C(\varphi)$ depends on the choice of direct sum decomposition; however every decomposition gives rise to the same ideal $C(\varphi)$; see also Proposition~\ref{NE}.\ref{NE2} below). 

\smallskip
For the convenience of the reader, we provide direct proofs of the more immediate properties of the ideal $C(\varphi)$; most notably we describe a generating set of this ideal in Proposition~\ref{formula}. The more difficult results about depth, unmixedness, Hilbert series, and resolutions, however, are merely restated 
from \cite{KPU-ann}; see Theorem~\ref{unm-part}.

\begin{lemma} \label{RE}Retain the setting of  Definition~{\rm \ref{C-of-phi}}. There exists a 
polynomial $h$ in $S$ so that the equality of row vectors 
$$ [F_1,\dots,F_{n+d}]=h\, \cdot [T_1,\ldots,T_n, -x_1, \ldots, -x_d]\, $$
obtains.
\end{lemma}
\begin{proof} We may assume that the row vector $\underline{F}=[F_1, \ldots, F_{n+d}]$ is not zero. Then $\goth B$ has rank at least $n+d-1$. The alternating property of $\varphi$ and  equality (\ref{FE}) imply
$$0=\underline{T}\cdot \varphi\cdot \underline{T}^{t}=\underline{x}\cdot B \cdot \underline{T}^{t}\, .$$
Since the matrix $B$ has only entries in the polynomial ring $T,$ we conclude that 
$$B \cdot \underline{T}^{t}=0 \quad \mbox{and} \quad \underline{T} \cdot B^t=0\, .$$
Hence 
$$ [T_1,\ldots,T_n, -x_1, \ldots, -x_d] \cdot \goth B=0\, .$$
Since also $$[F_1,\dots,F_{n+d}] \cdot \goth B=0$$
and $\goth B$ has almost maximal rank, there exists an element $h$ in the quotient field of $S$ such that 
$$[F_1,\dots,F_{n+d}]=h\, \cdot [T_1,\ldots,T_n, -x_1, \ldots, -x_d]\, .$$
The element $h$ is necessarily  in $S$ because the ideal $(T_1, \ldots, T_n, x_1, \ldots, x_d)$ of $S$ has grade at least two.
\end{proof}

Lemma \ref{RE} shows that either $h$ and the submaximal Paffians of $\goth B$ all vanish, or else, these elements are all
non-zero. In the latter case, $n+d$ has to be odd. Moreover, regarding $S$ as standard bi-graded with ${\rm deg} \, x_j =(1,0)$
and ${\rm deg} \, T_i= (0,1)$, we see that $\goth B$ is the matrix of a bi-homogeneous linear map. Thus the Paffians of this
matrix are bi-homogeneous, hence $h$ is bi-homogeneous, and computing bi-degrees one sees that
$$ {\rm deg} \, h = \left(\frac{n-d-1}2, \, d-1\right) .$$
In particular, if $h \neq 0$ then $d<n$. The next remark is now immediate.

\begin{remark}\label{trivialR}Retain the setting of Definition~\ref{C-of-phi}.
\begin{enumerate}[\rm(a)]
\item The ideal $C(\varphi)$ of $T$ is generated by  homogeneous forms  of degree $d-1$. 
\item\label{nt} If $n\le d$ or $n+d$ is even, then $C(\varphi)=0$.
\end{enumerate}
  \end{remark}

The next proposition shows that the ideal $C(\varphi)$ can be defined using any submaximal Pfaffian of $\goth B$; it also relates  $C(\varphi)$ to the socle modulo
the ideal $I_d(B)$ of $T$.
\begin{proposition}\label{NE}Retain the setting of Definition~\ref{C-of-phi}. 
 \begin{enumerate}
[\rm(a)]\item\label{NE1} 
$ c_T(F_i)=T_i \cdot c_T(F_{n+j})= T_i \cdot C(\varphi)  \ \,  \mbox{  for every } 1\le i \le n \ \mbox{ and }  1\le j \le d .$

\item\label{NE2}  $c_T(F_{n+j})=C(\varphi) \ \,   \mbox{  for every } 1\le j \le d .$
\item\label{NE3} $C(\varphi) \subset  I_d(B):_T(T_1,\dots,T_n)\, ; $ in other words, the image of $C(\varphi)$ is contained in the socle of the standard graded $k$-algebra $T/I_d(B)$. 
\end{enumerate}
\end{proposition}
\begin{proof} Lemma~\ref{RE} shows that 
$$c_T(F_i)=c_T(h \cdot T_i)= T_i \cdot c_T(h) \quad \mbox{   for } 1\le i \le n $$
and 
$$ \ \ \, c_T(F_{n+j})=c_T(h \cdot x_j)=  c_T(h) \quad  \,  \mbox{   for }  1\le j \le d\, .$$
This proves (\ref{NE1}). Part (\ref{NE2}) is an immediate consequence of (\ref{NE1}). 

Finally, the definition of   $F_i$ shows that  in the range $1 \le i \le n $,
$$F_i \in I_d(B)\cdot S \ \ \, \mbox{ and therefore } \ \ \,  c_T(F_i) \subset I_d(B) \,.$$
Now apply (\ref{NE1}) to deduce  (\ref{NE3}).
\end{proof}


 
 \medskip
  
In light of Remark \ref{trivialR}.b we now assume that $d<n$ and $n+d$ is odd. Fix an integer $i$ with $1 \leq i \leq n$. Let $J= \{ j_1, \ldots, j_d \}$ be a subset of $\{1, \ldots, n\} \setminus \{i\}$ with $j_1 < \ldots < j_s <i<j_{s+1} <\ldots< j_d$. Write $|J|=j_1+\ldots + j_d-d+s$. Let ${\rm Pf}_J(\varphi_i)$ be the Pfaffian of the matrix obtained from $\varphi$ by deleting rows and columns $i, j_1, \ldots, j_d$ and denote by $\Delta_J(B)$  the determinant of the $d\times d$ matrix consisting of columns $j_1, \ldots, j_d$ of $B$. 

Let $k[y_1, \ldots, y_d]$ be another polynomial ring  in $d$ variables that acts on the polynomial ring $T[x_1, \ldots,x_d]$ by contraction and let $\circ$ denote the contraction operation. 

For $M$ a monomial in $k[y_1, \ldots, y_d]$ of degree $\frac{n-d-1}2$ we define 
\begin{equation}\label{the f M} f_M=\frac{1}{T_i}\cdot \sum_J (-1)^{|J|} (M\circ {\rm Pf}_J(\varphi_i))\cdot \Delta_J(B)\, ,
\end{equation}
where $J$ ranges over all subsets of $\{1, \ldots, n\} \setminus \{i\}$ of cardinality $d$. Notice that, since ${\rm Pf}_J(\varphi_i) \in k[x_1, \ldots, x_d]$ is a homogeneous polynomial
of degree $\frac{n-d-1}{2}$, the element $M \circ {\rm Pf}_J(\varphi_i)\in k$ is simply the coefficient in this polynomial of the monomial in $k[x_1, \ldots, x_d]$ corresponding to $M$.
 
Similarly, fix an integer $j$ with $1 \leq j \leq d$. Let $J= \{ j_1, \ldots, j_{d-1} \}$ be a subset of $\{1, \ldots, n\}$ and set $||J||=j_1+\ldots + j_{d-1}$. Write ${\rm Pf}_J(\varphi)$ for
the Pfaffian of the matrix obtained from $\varphi$ by deleting rows and columns $j_1, \ldots, j_{d-1}$ and denote by $\Delta_J(B_j)$ the determinant of the matrix consisting of columns $j_1, \ldots, j_{d-1}$ of $B$ with row $j$ removed. Let $M \in k[y_1, \ldots, y_d]$ be a monomial of degree $\frac{n-d+1}2$ that is divisible by $y_j$. We define 
\begin{equation}\label{the h M} h_M= \sum_J (-1)^{||J||} (M\circ {\rm Pf}_J(\varphi))\cdot \Delta_J(B_j)\, ,
\end{equation}
where $J$ ranges over all subsets of $\{1, \ldots, n\}$ of cardinality $d-1$. Notice that $M \circ {\rm Pf}_J(\varphi)\in k$ is the coefficient in ${\rm Pf}_J(\varphi)$
of the monomial corresponding to $M$.

\begin{proposition}\label{formula} Adopt the setting of  Definition~{\rm \ref{C-of-phi}} with $d<n$ and $n+d$ odd. The
ideal $C(\varphi)$ of $T$ is generated by the elements $f_M$ of $\, (${\rm \ref{the f M}}$)$, where $M$ ranges over all monomials in $k[y_1, \ldots, y_d]$ of degree $\frac{n-d-1}{2}$;
it is also generated by the elements $h_M$ of $\, (${\rm \ref{the h M}}$)$, where $M$ ranges over all monomials of degree $\frac{n-d+1}{2}$ that are divisible by $y_j$.
\end{proposition}
\begin{proof} We prove the first claim. From Proposition~\ref{NE}.\ref{NE1} we know that $C(\varphi)=\frac{1}{T_i}\cdot c_T(F_i)$, for $1 \leq i \leq n$. Expanding the Pfaffian $F_i$ by maximal minors of $B$, one obtains 
$$
F_i=\pm \sum_J (-1)^{|J|}  \, {\rm Pf}_J(\varphi_i)\cdot \Delta_J(B)\, ,
$$
where, again, $J$ ranges over all subsets of $\{1, \ldots, n\} \setminus \{i\}$ of cardinality $d$; see for instance \cite[Lemma B.1]{IK}. 
Regarded as polynomials in $T[x_1, \ldots,x_d]$, the elements $\Delta_J(B)$ are constants and $F_i$ is homogeneous of degree $\frac{n-d-1}{2}$. Therefore, the content ideal $c_T(F_i)$ is generated by the elements
$$M \circ F_i= \pm \sum_J (-1)^{|J|} (M\circ {\rm Pf}_J(\varphi_i))\cdot \Delta_J(B)\, ,$$
where $M$ ranges over all monomials in $k[y_1, \ldots, y_d]$ of degree $\frac{n-d-1}{2}$. 

To prove the second claim we recall that $C(\varphi)=c_T(F_{n+j})$, for $1 \leq j \leq d$; see Proposition~\ref{NE}.\ref{NE2}. Now expand the Paffian $F_{n+j}$ along 
the last $d-1$ rows and use the fact that the monomials in the support of $F_{n+j} \in T[x_1, \ldots, x_d]$ have degree $\frac{n-d+1}{2}$ and are divisible by $x_j$; see
Lemma \ref{RE} and the discussion following it.
\end{proof}

\medskip

We illustrate the propositions above with an application to the ideal $I_d(B)+C(\varphi)$ in the case $n=d+1$; see \cite[2.10]{J} for a similar result.

\begin{example}
Adopt the setting of  Definition~{\rm \ref{C-of-phi}} with $n=d+1$.
Write $\Delta_i$ for the maximal minor of $B$ obtained by deleting column $i$. The first part of Proposition~\ref{formula} (or Proposition~\ref{NE}.\ref{NE1} and the fact
that $F_i=\pm \, \Delta_i$) show that $C(\varphi)$ is the principal ideal generated by 
$\frac{\Delta_i}{T_i}$, for any $i$ with $1\le i \le n$. Therefore $I_d(B)+C(\varphi)$ is the principal ideal generated by $\frac{\Delta_i}{T_i}$. 
\end{example}
\medskip

The next theorem is  \cite[\unmpart]{KPU-ann}. 

\begin{theorem}\label{unm-part}
Adopt the setting of  Definition~{\rm \ref{C-of-phi}} with $3\le d<n$ and $n$ odd. If $n-d\le\htt I_d(B)$, then the following statements hold.
\begin{enumerate}[\rm(a)]
\item\label{8.3.c} The ring $T/(I_d(B)+C(\varphi))$ is Cohen-Macaulay on the punctured spectrum and $$\operatorname{depth} (T/(I_d(B)+C(\varphi)))=\begin{cases}
1&\text{if $d$ is odd}\\
2&\text{if $d$ is even and $d+3\le n$} \\
n-1&\text{if $d+1=n$\,}.
\end{cases}$$
\item\label{8.3.a} The unmixed part of $I_d(B)$ is equal to
$I_d(B)^{\text{\rm unm}}=I_d(B)+C(\varphi)$.
\item\label{8.3.d} If $d$ is odd, then the  ideal  $I_d(B)$ is unmixed;
furthermore, $I_d(B)$ has a linear free resolution and $\operatorname{reg}(T/I_d(B))=d-1$.
\item\label{8.3.e} The Hilbert series of $\, T/(I_d(B)+C(\varphi))$ is 
\begin{align*}
\phantom{+}&\frac{\sum\limits_{\ell=0}^{d-2}\binom{\ell+n-d-1}{n-d-1}z^\ell + \sum\limits_{\ell=0}^{n-d-2}(-1)^{\ell+d+1}\binom{\ell+d-1}{d-1} z^{\ell+2d-n}}{(1-z)^d}
+(-1)^d\sum\limits_{j\le \lceil\frac{n-d-3}{2}\rceil}
\binom{j+d-1}{d-1}z^{2j+2d-n}  .
\end{align*}
\item\label{8.3.b} The multiplicity of 
$\, T/(I_d(B)+C(\varphi))$ is 
$$\sum\limits_{i=0}^{\lfloor \frac{n-d}{2}\rfloor}\binom{n-2-2i}{d-2}.$$
\end{enumerate}\end{theorem}

\medskip
\begin{Remark}The conclusions of assertion (\ref{8.3.d}) do not hold when $d$ is even: the ideal $I_d(B)$ can be mixed and the minimal homogeneous resolution of $I_d(B)^{\rm unm}$ may not be linear; see, for example, \cite[(\notlinear)]{KPU-ann}. 
\end{Remark}

\begin{corollary}\label{content} Adopt Data~\ref{data1}. Further assume that $I$ is a linearly presented Gorenstein ideal of height three. Let $\varphi$   be a minimal homogeneous alternating presentation matrix for $g_1, \ldots, g_n$,  $B$ be the Jacobian dual of $\varphi$ in the sense of {\rm\ref{def1}}, and $C(\varphi)$ be the content ideal of  Definition~{\rm\ref{C-of-phi}}. Then 
$$I_d(B)+C(\varphi)\subset I_d(B):_T(T_1,\dots,T_n)\subset I(X).$$
In particular, $I_d(B)+C(\varphi)$ is contained in the defining ideal $\J$ of the Rees ring of $I$. 
\end{corollary}
\begin{proof} The first containment follows from Proposition~\ref{NE}.\ref{NE3}. Since $I$ is linearly presented, $I_d(B)$ is an ideal of $T$ and therefore $I_d(B)\subset I(X)$ according to (\ref{6.2.1}) and (\ref{I of X}).  
Now the second containment follows because $I(X)$ is a prime ideal containing $I_d(B)$ but not $(T_1, \ldots, T_n)$. 
\end{proof}

\smallskip

\section{Defining equations of blowup algebras of linearly presented height three Gorenstein ideals}
\label{grade-3-gor}

\smallskip

In this section we assemble the proof of the main result of the paper, Theorem~\ref{MainT}. We  also provide a more detailed version of the theorem than the one stated in the Introduction; in particular, we express the relevant defining ideals as socles and iterated socles.

\begin{theorem}\label{MainT}
Let   $R=k[x_1,\dots,x_d]$ be  a  polynomial ring over a field $k$,  $I$ be a linearly presented  Gorenstein ideal of height three which is  minimally generated by  homogeneous forms $g_1, \ldots, g_n$, 
 $\varphi$ be a minimal homogeneous alternating presentation matrix for 
$g_1, \ldots, g_n$,  $B$ be the Jacobian dual of $\varphi$ in the sense of {\rm\ref{def1}}, $C(\varphi)$ be the ideal of $T=k[T_1, \ldots, T_n]$ given in Definition~{\rm\ref{C-of-phi}}, and $\mathcal L$ be the ideal of $S=R[T_1,\dots,T_n]$ which defines $\Sym(I)$ as described in {\rm\ref{def1}}.
If $I$ satisfies $G_d$,
then the following statements hold.
\begin{enumerate}[\rm(a)] \item\label{MainT-a} The  ideal of $S$ defining the Rees ring of $\, I$ is 
$$\J=\mathcal L+ I_d(B)S+C(\varphi)S\,.$$
If  $d$ is odd, then  $C(\varphi)$ is zero and $\J$ has the expected form.
 \item\label{MainT-b} The  ideal of $\, T$ defining the variety $X$ parametrized by $g_1, \ldots, g_n$  is 
 $$I(X)=I_d(B)+C(\varphi)\, .$$ 
 \item\label{MainT-d}  
 The two ideals $\J$ and $\L:_S (x_1, \ldots, x_d)(T_1, \ldots,T_n)$ of $\, S$ are equal. 
\item\label{MainT-c} The three ideals    $I(X)$,
$I_d(B):_T(T_1, \ldots,T_n)$, and $I_d(B)^{\text{\rm unm}}$  of $\, T$ are equal. 
\end{enumerate}
\end{theorem}
\begin{proof}  From   Corollary~\ref{content} we have the containment  
\begin{equation}\label{cont} I_d(B)+C(\varphi)\subset I(X)\,.\end{equation}
If $n\le d$, then 
$\J=\mathcal L$ by (\ref{Aug232016}) and hence
$I(X)=0$ by (\ref{I of X}), which also gives $I_d(B) + C(\varphi)=0$. Therefore we may assume that $d<n$.  

Corollary~\ref{fiber type} guarantees that
the ideal $I$ is of fiber type, which means that $\J=\mathcal L+I(X)S$; see \ref{LT-FT}.  Thus (\ref{MainT-a}) 
follows from part (\ref{MainT-b}) (and Remark \ref{trivialR}.b).

According to Theorem \ref{MULT}  the homogeneous coordinate ring $A=T/I(X)$ of $X$ has dimension $d$. Now Remark~\ref{ass-satisfied} and Theorem \ref{correct grade}.\ref{Y2.8.2} show that 
\begin{equation}\label{MTf}\htt I_d(B)=\htt I(X)=n-d\,.\end{equation}
The hypothesis of Theorem~\ref{unm-part} is satisfied; hence assertions (\ref{8.3.a}) and (\ref{8.3.b}) of Theorem~\ref{unm-part} yield 
\begin{equation}\label{MTe}I_d(B)^{\text{\rm unm}}=I_d(B)+C(\varphi),\end{equation}
and the multiplicity of 
$T/(I_d(B)+C(\varphi))$ is 
$$\sum\limits_{i=0}^{\lfloor \frac{n-d}{2}\rfloor}\binom{n-2-2i}{d-2}.$$
By Theorem \ref{MULT}, the coordinate ring $A=T/I(X)$ has the same multiplicity;  by (\ref{cont}) and (\ref{MTf}), the rings  $T/(I_d(B)+C(\varphi))$ and $T/I(X)$ have the same dimension; and by (\ref{MTe}), both rings are unmixed. Hence the containment $I_d(B)+C(\varphi)\subset I(X)$ of (\ref{cont}) is an equality,  and (\ref{MainT-b}) is also established. 

Part (\ref{MainT-c}) follows from Corollary~\ref{content}, part (\ref{MainT-b}), and (\ref{MTe}).  For the proof of (\ref{MainT-d}) notice that parts (\ref{MainT-a}), (\ref{MainT-b}),  and (\ref{MainT-c}) give 
$$\J=\L+(I_d(B):_T(T_1, \ldots,T_n))S\, .$$
Therefore, 
$$\J \subset (\L+I_d(B)S):_S(T_1, \ldots,T_n)\,.$$
Since $\L+I_d(B)S \subset \L:_S(x_1, \ldots, x_d)$  according to (\ref{6.2.1}), we obtain
$$\J\subset (\L:_S(x_1, \ldots, x_d)):_S(T_1, \ldots,T_n)=\L:_S (x_1, \ldots, x_d)(T_1, \ldots,T_n)\subset \J\, .$$
\end{proof}

\begin{remark}\label{general setting} 
Adopt Data~\ref{data1}. The assumption in Theorem~\ref{MainT} that $I$ is linearly presented can be weakened to the condition that 
the entries of a minimal homogeneous presentation matrix $\varphi$ of $I$ generate a complete intersection ideal. 

Indeed, let $\varphi$ be a minimal homogeneous alternating presentation matrix of $g_1, \ldots, g_n$. By the symmetry of the minimal homogeneous $R$-resolution 
of $R/I$ all entries of $\varphi$ have the same degree $D$. 
Let $y_1, \ldots, y_s$ be a regular sequence of homogeneous forms of degree $D$ that generate $I_1(\varphi)$. These forms are algebraically independent over $k$, and $R=k[x_1, \ldots,x_d]$ is flat over the polynomial ring $k[y_1, \ldots,y_s]$. The entries of $\varphi$ are forms of degree $D$ and are $R$-linear combinations of $y_1, \ldots,y_s$; hence these entries are linear forms in the ring $k[y_1, \ldots,y_s]$. Since $g_1, \ldots, g_n$ are signed submaximal Paffians of $\varphi$ according to \cite{BE}, these elements also
belong to  the ring $k[y_1, \ldots,y_s]$, and by flat descent they satisfy the assumptions of Theorem~\ref{MainT} as elements of this ring.
Now apply Theorem~\ref{MainT}.
Flat base change then gives the statements about the Rees ring over $R$, with $y_1, \ldots, y_s$ in place of $x_1, \ldots, x_d$; the statements about the ideal $I(X)$ are independent of the ambient ring. 
\end{remark}

\smallskip

Paper \cite{KPU-ann} was  written with the intention of understanding the ideals  $I_d(B)^{\rm unm}$ in order to determine the equations defining the Rees algebra and the special fiber ring of height three Gorenstein ideals, the ultimate goal of the present paper. In particular, in  \cite{KPU-ann} the ideals $I_d(B)^{\rm unm}$ have been resolved.  
Now that we have proven that  $I_d(B)^{\rm unm}$ defines the special fiber ring, 
we are able to harvest much information from the results of \cite{KPU-ann}.

\begin{corollary}\label{Cor-to-main} Adopt the notation  and hypotheses  of Theorem~{\rm \ref{MainT}}. The following statements hold.
\begin{enumerate}[{\rm (a)}]
\item The variety $X$ is Cohen-Macaulay. 
\item If $n\le d+1$, then the  homogeneous coordinate ring $A$ of $\, X$ is Cohen-Macaulay. Otherwise
$$\operatorname{depth}\, A=\begin{cases}
1&\text{if $d$ is odd }\\
2&\text{if $d$ is even\,.} 
\end{cases}$$
\item If $d$ is odd and $d<n$, then $I(X)$ has a linear free resolution and $\operatorname{reg}(A)=d-1$.
\item If $d<n$, then the Hilbert series of $A$ is 
\begin{align*}
\phantom{+}&\frac{\sum\limits_{\ell=0}^{d-2}\binom{\ell+n-d-1}{n-d-1}z^\ell + \sum\limits_{\ell=0}^{n-d-2}(-1)^{\ell+d+1}\binom{\ell+d-1}{d-1} z^{\ell+2d-n}}{(1-z)^d}
+(-1)^d\sum\limits_{j\le \lceil\frac{n-d-3}{2}\rceil}
\binom{j+d-1}{d-1}z^{2j+2d-n}  .
\end{align*}
\end{enumerate}
\end{corollary}
\begin{proof} We may assume that $d<n$ since otherwise $I(X)=0 \, ;$ see (\ref{Aug232016}) and (\ref{I of X}). From 
Theorem~\ref{MainT}.b we know that $I(X)=I_d(B)+C(\varphi)$, where $C(\varphi)=0$ if $d$ is odd.
Thus the assertions are items (\ref{8.3.c}), (\ref{8.3.d}), and (\ref{8.3.e}) of Theorem~\ref{unm-part}, which
applies due to (\ref{MTf}).
\end{proof}

\begin{Remark} The  minimal homogeneous resolution of the 
ideal $I(X)$ may not be linear when $d$ is even; see \cite[(\notlinear)]{KPU-ann}.\end{Remark}

\end{document}